\documentclass[12pt,hidelinks]{article}

\title{The Boolean Compactness Theorem for $\mathrm{L}_{\infty\infty}$}
\author{Juan M Santiago Suárez, Matteo Viale}
\date{\today}

\usepackage[a4paper]{geometry}
\usepackage[dvipsnames]{xcolor}
\usepackage{amsmath,amsthm}
\usepackage{amssymb,amsfonts}
\usepackage{mathtools}
\usepackage{stmaryrd}
\usepackage{hyperref}
\usepackage{tikz}
	\usetikzlibrary{cd}
	\usetikzlibrary{calc}
\usepackage{pifont}
\usepackage[shortlabels]{enumitem}
\usepackage{bm}
\usepackage{cite}
\usepackage{graphicx}
\usepackage{verbatim}
\usepackage{tipa}
\usepackage{prftree}
\usepackage{titlesec}
\usepackage{float}
\usetikzlibrary{arrows.meta, positioning}

\usepackage[mode=multiuser,status=draft,lang=english]{fixme}
\fxsetup{theme=colorsig}

\FXRegisterAuthor{MV}{aM}{Matteo}
\FXRegisterAuthor{JMSS}{aJ}{Juan}

\titlespacing*{\section}
{0pt}{0,5cm}{0,5cm}

\newcommand{\Linf}{\ensuremath{\mathrm{L}_{\infty \omega}}}
\newcommand{\Linff}{\ensuremath{\mathrm{L}_{\infty \infty}}}

\theoremstyle{plain}
	\newtheorem{Theorem}{Theorem}

	\newtheorem{theorem}{Theorem}[section]
	\newtheorem{proposition}[theorem]{Proposition}
	\newtheorem{lemma}[theorem]{Lemma}
	\newtheorem{corollary}[theorem]{Corollary}
	\newtheorem{fact}[theorem]{Fact}

	\newtheorem{question}[theorem]{Question}

\theoremstyle{definition}
	\newtheorem{Definition}{Definition}
	\newtheorem{definition}[theorem]{Definition}
	\newtheorem{notation}[theorem]{Notation}
	\newtheorem{example}[theorem]{Example}

\theoremstyle{remark}
	\newtheorem{remark}[theorem]{Remark}

\DeclareMathOperator{\RO}{RO}

\newcommand{\bool}[1]{\mathsf{#1}}

\newcommand{\Qp}[1]{\left\llbracket #1 \right\rrbracket}

\newcommand{\bp}[1]{\left\lbrace #1 \right\rbrace}

\newcommand{\rao} {\rightarrow}

\newcommand{\Lrao} {\Leftrightarrow}
\newcommand{\Rao} {\Rightarrow}

\newcommand{\MM}{\ensuremath{\text{{\sf MM}}}}

\begin{document}

 \maketitle

\begin{abstract}
We show that, contrary to the commonly held view, there is a natural and optimal compactness theorem for 
$\mathrm{L}_{\infty\infty}$ which generalizes the usual compactness theorem for first order logic. The key to this result is the switch from Tarski semantics to Boolean valued semantics. On the way to prove it, we also show that the latter is a (the?) natural semantics both for $\mathrm{L}_{\infty\infty}$ and for $\mathrm{L}_{\infty\omega}$.
\end{abstract}

\thanks{
The second author acknowledges support from INDAM through GNSAGA and from the project: \emph{PRIN 2022 Models, sets and classifications, prot. 2022TECZJA}.
The first author acknowledges support from the project \emph{VINCI 2021 Chapitre 2 - AIDES A LA MOBILITE POUR THESES EN COTUTELLE C2-99}.


We thank Boban Velickovic (who outlined us the relevance of consistency properties), Ben de Bondt (for many useful comments), and the referee.}

\tableofcontents

\section{Introduction}

The compactness theorem for first order logic is the cornerstone result over which the machinery of model theory develops. In the current paper we generalize this result to its optimal version for the infinitary logics $\mathrm{L}_{\infty\infty}$ and $\mathrm{L}_{\infty\omega}$. This has to be taken with a grain of salt, as it is well known that the literal generalization of the compactness theorem to these logics is false (see Example \ref{faicom}). 
To overcome this obstacle we revise the semantics and move from Tarski semantics to \emph{Boolean valued} semantics (see Section \ref{subsec:boolvalsem} for details); we also revise appropriately the notion of finitely consistent theory  to that of \emph{finitely conservative} theory (see the key Def. \ref{def:finitecons0}). Theorem \ref{thm:boolcomp0} shows that any $\mathrm{L}_{\infty\infty}$-theory is consistent for Boolean valued semantics (from now on Boolean consistency) if and only if it is logically equivalent to a finitely conservative theory. 
Hence, finite conservativity stands to Boolean consistency for $\mathrm{L}_{\infty\infty}$ much alike finite consistency stands to consistency for first order logic. In fact, in a very precise sense, we can show that Thm. \ref{thm:boolcomp0} is a generalization of first order compactness (see Theorem \ref{thm:focompmain}).

We already showed that Boolean valued models provide an appropriate semantics for infinitary logics in \cite{MatteoJuanBoolean} (where we established natural versions of completeness, interpolation, Beth-definability, and omitting types for the infinitary logics $\mathrm{L}_{\infty\infty}$ and $\Linf$ relative to this semantics). The current paper closes the ring adding on top of the above results this optimal compactness theorem for $\mathrm{L}_{\infty\infty}$ (and $\Linf$) relative to this semantics. 

We will expand on the comparison between Tarski semantics and Boolean valued semantics and on the merits of the latter in Section \ref{sec:boolvalsemvstarskisem}. Let us rightaway note that the naturalness and usefulness of the latter follows already from the following observations: a Tarski consistent (i.e. admitting a Tarski model) theory is also Boolean consistent (while the converse fails in general for uncountable $\mathrm{L}_{\infty\infty}$-theories); for countable $\mathrm{L}_{\omega_1\omega}$-theories, Tarski consistency, Boolean consistency, and finite conservativity are equivalent properties (up to logical equivalence).

A strong point of the present paper (which to a large extent applies also to \cite{MatteoJuanBoolean}) is (for us) that we can obtain all these results by elementary means: the proofs require tools available to any scholar who has had a general course in logic covering the usual results for first order logic (completeness, compactness) and has some familiarity with Boolean algebras and compactness arguments in general topology.\footnote{Some of the results in \cite{MatteoJuanBoolean} on the other hand require a certain familiarity with the forcing method in set theory.}
As already outlined, the main idea guiding our work (here and in \cite{MatteoJuanBoolean}) is the shift from Tarski semantics to Boolean valued semantics in the analysis of infinitary logics. Once one gets accustomed to the different semantical framework, all proofs flow naturally as the expected outcome of the right definitions.
Note also that Boolean valued models compare to ordinary Tarski models, much alike (pre)sheaves of rings compare to rings, hence the semantical shift we perform here and in \cite{MatteoJuanBoolean} is natural, as it reflects in the field of infinitary logics the same pattern which led algebraic geometers to the development of sheaf theory in their quest for the most general formulation of the notion of algebraic variety.

We will detail with more precision the above scant general considerations in Section \ref{sec:boolvalsemvstarskisem}.\medskip

The infinitary logic $\mathrm{L}_{\kappa\lambda}$ admits as formulae those constructed from the atomic formulae of the signature $\mathrm{L}$ using negation, conjunctions and disjunctions of size less than $\kappa$, and blocks  of quantifiers $\exists(x_i:i<\gamma)$, $\forall(x_i:i<\gamma)$  on a string of variables indexed by some $\gamma<\lambda$. The $\mathrm{L}_{\kappa\lambda}$-formulae can only have less than $\lambda$-many free variables.\footnote{Hence infinitary disjunctions/conjunctions on a set of $\mathrm{L}_{\kappa\lambda}$-formulae are allowed only if the formulae have all their free variables occurring in a fixed set of size less than $\lambda$.}

The logic $\mathrm{L}_{\infty\lambda}$ is the union of the logics $\mathrm{L}_{\kappa\lambda}$ for $\kappa$ a cardinal; the logic $\mathrm{L}_{\infty\infty}$ is the union of the logics $\mathrm{L}_{\kappa\lambda}$ for $\kappa,\lambda$ cardinals. Our focus in this paper is mainly on the logics $\mathrm{L}_{\infty\omega}$ and $\mathrm{L}_{\infty\infty}$.

We now state the main definitions and results, i.e. the compactness theorem for $\mathrm{L}_{\infty\infty}$ and its corollaries. This requires to revise the notion of finitely consistent theory, which is central in the standard formulation of the compactness theorem for first order logic. The following two definitions introduce the key concepts leading to a revised notion of finitary consistency.
\begin{Definition} \label{Def:consstrength}
Let $\psi_0,\psi_1$ be $\mathrm{L}_{\infty\infty}$-sentences such that $\psi_1\vdash \psi_0$.\footnote{$\vdash$ refers to the natural generalization to $\mathrm{L}_{\infty\infty}$ of the usual notion of provability, see Section \ref{subsec:proofsystem} below for its formal definition.} 
$\psi_1$ is a \emph{conservative strengthening of $\psi_0$} if for all finite sets $s$ of subsentences\footnote{Note that by a subsentences of an $\mathrm{L}$-sentence $\psi$ we mean any sentence obtained from a subformula of $\psi$ by substituting its free variables with constants in (an arbitrary expansion of) $\mathrm{L}$, see for details Notation \ref{not:subformulaeconprop}.} of $\psi_0$, $\psi_0\wedge\bigwedge s$ is Boolean consistent if and only if so is $\psi_1\wedge\bigwedge s$.
\end{Definition}

Essentially, the above definition isolates those strengthenings of $\psi_0$ which are not able to tell apart from $\psi_0$ those properties expressible in the language of $\psi_0$ using subsentences of $\psi_0$ which are consistent with $\psi_0$.

Finite conservativity defined below is the new central concept  replacing finite consistency in the formulation of the compactness theorem for $\mathrm{L}_{\infty\infty}$.

\begin{Definition}\label{def:finitecons0}
Let $\{\psi_i:i\in I\}$ be a family of $\mathrm{L}_{\infty\infty}$-sentences. $\{\psi_i:i\in I\}$ is \emph{finitely conservative} if:
\begin{itemize} 
\item
at least one $\psi_i$ is Boolean consistent (see Definition \ref{def:BolCon}), and
\item
for all $u\subseteq v$ non-empty finite subsets of $I$, $\bigwedge_{i\in v}\psi_i$ is a conservative strengthening of $\bigwedge_{i\in u}\psi_i$ .
\end{itemize}
\end{Definition}

A trivial observation is the following: $T$ is Boolean consistent if and only if $\bp{\bigwedge T}$ is finitely conservative, henceforth any Boolean consistent theory is logically equivalent to a finitely conservative one.

Note that the following family of sentences is a finitely consistent (but not finitely conservative) $\mathrm{L}_{\omega_1\omega}$-theory which is not Boolean consistent:
\[
\mathrm{T} = \{ \bigvee_{n < \omega} c_\omega = c_n \}\cup \ \{c_n \neq c_m : n <m\leq \omega \}.
\]	
Let $v = \{\bigvee_{n < \omega} c_\omega = c_n, c_0 \neq c_\omega\}$, and $u=\bp{\bigvee_{n < \omega} c_\omega = c_n}$, then $c_0 = c_\omega$
is consistent with $\bigvee_{n < \omega} c_\omega = c_n$ (which is $\bigwedge u$), but not with $\bigwedge v$, hence $\bigwedge v$ is not a conservative strengthening of $\bigwedge u$ (see Example \ref{faicom} for details).

We can now state the Boolean Compactness theorem for $\mathrm{L}_{\infty\infty}$:
 \begin{Theorem}[Boolean Compactness Theorem for $\mathrm{L}_{\infty\infty}$]\label{thm:boolcomp0}
      Let $\{\psi_i:i\in I\}$ be a finitely conservative family of $\mathrm{L}_{\infty\infty}$-sentences.

Then $\bigwedge_{i \in I} \psi_i$ is Boolean consistent and is a conservative strenghtening of $\bigwedge_{i\in u}\psi_i$ for all $u$ non-empty and finite subset of $I$.

Furthermore, if each $\psi_i$ is an $\mathrm{L}_{\infty\omega}$-sentence, 
$\bigwedge_{i \in I} \psi_i$  has a mixing model (see Definition \ref{def:mixingmodel}).\footnote{See Section \ref{sec:boolvalsemvstarskisem} to appreciate the difference between being Boolean consistent and admitting a mixing model witnessing it.}

	\end{Theorem}
We already noticed that any Boolean consistent theory $T$ is clearly the logical consequence of the theory
$\bp{\bigwedge T}$, which is easily seen to be finitely conservative. Hence,
an $\mathrm{L}_{\infty\infty}$-theory is Boolean consistent if and only if it is logically equivalent to a finitely conservative one.

The theorem generalizes first order compactness in view of the following (combined with Observation \ref{observation2} of Section \ref{sec:boolvalsemvstarskisem} yielding that Boolean consistent first order theories have a Tarski model):

\begin{Theorem}\label{thm:focompmain}
        Let $T$ be a finitely consistent first order theory. Then there exists an $\Linf$-theory $T^*$ such that :
        \begin{itemize}
        	\item $T^*$ is provably stronger than $T$, and 
        	\item $T^*$ is finitely conservative.
        \end{itemize}
        
	\end{Theorem}

We can also outline in which ways the Boolean Compactness Theorem is related to Barwise compactness (see Section \ref{sec:opeque} for details), and exemplify how the notion of finite conservativity can be used to construct Boolean valued models which realize certain types and omit others (see Section \ref{subsec:appcomp} or \cite{SantiagoPhd} for details).

The rest of the paper is organized as follows:
\begin{itemize}
\item
The introductory Section \ref{sec:boolvalsemvstarskisem} gives a detailed comparison of the Boolean valued semantics and Tarski semantics for infinitary logics and of their merits and limits. Its content is not needed to follow the technical parts of the paper, but can serve as motivation and intuition for it. It is mathematically very light.
\item
Section \ref{sec:preliminaries} contains the necessary background material on infinitary logics, proof systems for them, Boolean valued models for them, and Consistency Properties (which are the tools to construct Boolean valued models of provably consistent infinitary theories).
\item
Section \ref{sec:compactness} contains the proof of the Boolean Compactness Theorem for $\mathrm{L}_{\infty\infty}$. It also presents a non-trivial application of the theorem. 
\item
Section \ref{FinConBooConTarCon} connects the above result to first order compactness and Barwise compactness. In particular, it shows that Tarski consistent $\mathrm{L}_{\kappa\lambda}$-theories, and syntactically complete $\mathrm{L}_{\kappa\lambda}$-theories are finitely conservative, and leverages on these results to give a proof of the first order compactness theorem which interpolates through the Boolean compactness theorem.
We also pose some questions on the relation between finite conservativity and Boolean consistency which we leave open.
\end{itemize}

\section{Tarski semantics vs Boolean valued semantics}\label{sec:boolvalsemvstarskisem}

This section serves only as a general motivation for the approach we take here and in \cite{MatteoJuanBoolean}
to the study of infinitary logics. Its content is not needed for the remainder of the paper.

The Tarski semantics for the sublogics of $\mathrm{L}_{\infty\infty}$ in an $\mathrm{L}$-stucture $\mathfrak{M}$ is defined as expected,  i.e.: atomic formulae get their assignment according to the usual rules, $\forall(x_i:i<\gamma)\,\psi$ holds in $\mathfrak{M}$ with assignment $\nu$ if for all $\vec{c}\in\mathfrak{M}^\gamma$, the formula $\psi$ holds in $\mathfrak{M}$ with the assignment $\nu_{\vec{c}}$ obtained from $\nu$ by replacing $\nu(x_i)$ with $\vec{c}(i)$ for all $i<\gamma$; $\bigwedge_{i\in I}\psi_i$ holds in $\mathfrak{M}$
with assignment $\nu$ 
if and only all the $\psi_i$ hold in $\mathfrak{M}$
with assignment $\nu$; accordingly one defines the semantics for the other logical constants.

Any usual proof system for first order logic can be naturally generalized to $\mathrm{L}_{\infty\infty}$. For example see the Gentzen type proof system we present below in Section \ref{subsec:proofsystem}.
It is  straightforward to establish the correctness theorem relative to the Tarski semantics for $\mathrm{L}_{\infty\infty}$, i.e.:
if there is a proof of the sequent $\Gamma\vdash\Delta$, then any $\mathrm{L}$-structure which realizes all formulae in $\Gamma$, realizes at least one formula in $\Delta$.
It can also be checked that the statement \emph{($P$ is a proof of $\Gamma\vdash\Delta$)}
is $\Delta_1(\mathsf{ZF})$.\footnote{For those unfamilliar with the notion of $\Delta_1(T)$-property, see \cite[Def. A.1.8]{VialeForcing} (alternatively see also \cite[Lemma 13.10]{JECHST}).} By a deeper result of Mansfield (the completeness theorem of this proof system for Boolean valued semantics \cite{MansfieldConPro}, see also below), it is also the case that the assertion \emph{($\Gamma\vdash\Delta$ has a proof)} is $\Delta_1(\mathsf{ZFC})$. This brings to light that this notion of proof for $\mathrm{L}_{\infty\infty}$ is set-theoretically simple: for example $\mathrm{L}_{\infty\infty}$-provable sequents form an absolutely definable class with respect to transitive models $M\subseteq N$ of $\mathsf{ZFC}$. 

However, a significant distinction between $\mathrm{L}_{\infty\infty}$ and $\mathrm{L}_{\infty\omega}$ occurs at the level of their Tarski semantics: for an $\mathrm{L}_{\infty\omega}$-sentence $\psi$, and an $\mathrm{L}$-structure $\mathfrak{M}$, the assertion 
$(\mathfrak{M}\models\psi)$ is $\Delta_1(\mathsf{ZF})$ in the parameters $\mathfrak{M},\mathrm{L},\psi$, hence absolute between transitive models of set theory having these relevant parameters.
The same does not occur for $\mathrm{L}_{\infty\infty}$: it is not hard to cook up a signature $\mathrm{L}$, an 
$\mathrm{L}_{\omega_1\omega_1}$-formula $\psi(x_n:n\in\omega)$ in displayed free variables, and an $\mathrm{L}$-structure $\mathfrak{M}$ such that the set theoretic universe $V$ models $\forall(x_n:n\in\omega)\,\psi(x_n:n\in\omega)$ holds in $\mathfrak{M}$, while the same assertion fails in a generic extension of $V$.\footnote{The sentence could formalize the notion of well order with uncountable cofinality, i.e its models are ordinals of uncountable cofinality.}

Another major issue of Tarski semantics is that the completeness theorem fails for the natural proof systems for $\mathrm{L}_{\infty\infty}$ relative to it; specifically there is an $\mathrm{L}_{\omega_2\omega}$-sentence $\psi$ which is provably consistent but does not hold in any Tarski model (in a suitable signature $\mathrm{L}$, $\psi$ asserts that there is a bijection of $\omega$ onto $\omega_1^V$).

So one is left with two options: either change the semantics or change the proof system.
It is the case that the main developments in the analysis of (the infinitary sublogics of)  $\mathrm{L}_{\infty\infty}$
followed the second path, i.e. most authors sticked to Tarski semantics, and either did not pay particular attention to syntactic proofs, or devised new proof systems which yield a completeness theorem for this semantics (e.g. this is the approach taken in \cite{MalitzThesis,MarkerInfLog}). Even if the advantages of Tarski semantics in the interpretation of meanings are notable, there are major drawbacks of this approach:
\begin{itemize}
\item The interpolation theorem fails \cite[Thm. 3.2.4]{MalitzThesis}. Note however that a major result of Shelah on infinitary logics is the isolation of a nice famiy of sublogics of $\mathrm{L}_{\infty\infty}$ for which the interpolation theorem holds relative to Tarski semantics \cite{SHE-INFLOG2012}.
\item The completeness theorem for Tarski semantics relative to natural proof systems fails. It
holds only relative to a proof system whose notion of proof is not absolute  between transitive models of $\mathsf{ZFC}$ \cite{MalitzThesis}.
\item The compactness theorem fails even for $\mathrm{L}_{\omega_1\omega}$, the smallest non-trivial fragment of $\mathrm{L}_{\infty\infty}$ (see Example \ref{faicom}). 
It is worth here to delve more into the compactness-like properties of sublogics of $\mathrm{L}_{\infty\infty}$ relative to Tarski semantics, as the picture is  variegate and there are various interesting sublogics of $\mathrm{L}_{\infty\infty}$ for which variations of compactness principles hold relative to this semantics. 
Limiting our attention to the main results stemming out of the work in the set theoretic and model theoretic community, we can mention several compactness theorems related to Tarski semantics. Let us recall Barwise compactness,  and a compactness result for $\omega$-logic due to Kreisel (and derivable from Barwise compactness): an overview of these results can be found in \cite[Thms. 3.1.1–3.1.5]{KeislerKnight}. In set theory, weakly and strongly compact cardinals emerged as natural generalizations of compactness for $\mathrm{L}_{\kappa\kappa}$ \cite[Ch. 4]{kanamori:higher_infinite} (whereas other large cardinal notions exhibit incompactness \cite{HanfIncompactness}). Furthermore, the study of generalized quantifiers has yielded additional compact logics \cite{ShelahCompact} (also for the logics of \cite{SHE-INFLOG2012} and for $\mathrm{L}_{\infty\omega}$ a weak analogue of compactness holds, namely the non-definability of well orders). More recently, compactness principles for omitting types have resurfaced as tools for characterizing very large cardinals \cite{BendaCompactness,BoneyOmit,BoneyOmit2}. 
However, compactness results for  infinitary sublogics of $\mathrm{L}_{\infty\infty}$ relative to Tarski semantics are sparse: it is most often the case that interesting infinitary logics are non-compact for this semantics (e.g. this is the case for $\mathrm{L}_{\infty\infty}$ and  $\mathrm{L}_{\infty\omega}$). In order to obtain "compact" logics for Tarski semantics, one has either to narrow the scope of the logics (e.g.  see \cite{KeislerKnight,ShelahCompact}) or revise/weaken the notion of compactness (e.g. -in one direction- the characterization of large cardinals in terms of compactness-like properties of $\mathrm{L}_{\kappa\kappa}$, as in \cite{kanamori:higher_infinite,BendaCompactness,BoneyOmit,BoneyOmit2}, or -in another direction- the non-definability of well orders, as in \cite{SHE-INFLOG2012}).
\end{itemize}

We address the  discrepancy between the proof system which is the straigthforward generalization to $\mathrm{L}_{\infty\infty}$ of Gentzen's sequent calculus (as the one given below in Section \ref{subsec:proofsystem}) and  Tarski semantics for $\mathrm{L}_{\infty\infty}$ following the alternative pattern of 
using Boolean valued semantics: roughly for a (complete) Boolean algebra $\bool{B}$ a $\bool{B}$-valued model $\mathcal{M}$ with domain $M$ for a signature $\mathrm{L}$ is a $\bool{B}$-valued 
$\mathrm{L}$-structure, where the truth values of formulae with parameters in $M$ are taken in a Boolean algebra $\bool{B}$ (for details see Def. \ref{def:Bvalmod}).
The delicate part is to assign consistently truth values $\Qp{\phi}$ in $\bool{B}$ to the atomic formulae; once this is done, we can define $\Qp{\neg\phi}=\neg_\bool{B}\Qp{\phi}$, $\Qp{\exists x\,\phi(x)}=\bigvee_{a\in M}\Qp{\phi(a)}$,  $\Qp{\bigvee_{i\in I}\phi_i}=\bigvee_{i\in i}\Qp{\phi_i}$
(for details see Def. \ref{def:boolvalsem}).
It is usually convenient/necessary to assume that $\bool{B}$ is complete to compute truth values for quantifiers and infinitary connectives. 

This approach has already been pursued by Karp and rediscovered by Mansfield \cite{Karp,MansfieldConPro}, who proved a
completeness theorem for the proof system we present in Section \ref{subsec:proofsystem}
relative to Boolean valued semantics.
Apart from Mansfield and Karp's contributions, and Barwise and Ellentuk's contributions \cite{Barwise1973back, Ellentuck1976categoricity}, we haven't traced in the literature many significant uses of Boolean valued semantics in the analysis of $\mathrm{L}_{\infty\infty}$ prior to our results.

In \cite{MatteoJuanBoolean} we proved that for the Boolean valued semantics one recovers also the interpolation theorem \cite[Thm. 3.2]{MatteoJuanBoolean}, the Beth-definability property \cite[Thm. 3.4]{MatteoJuanBoolean}, and the omitting types theorem \cite[Thm. 3.8]{MatteoJuanBoolean}. Furthermore, if one focusses on $\mathrm{L}_{\infty\omega}$, a class of Boolean valued models furnishing a complete semantics for this logic is given by the \emph{mixing models} defined below (see Def. \ref{def:mixingmodel}) \cite[Thm. 3.1]{MatteoJuanBoolean}. It has to be noted 
that, for the $\mathsf{B}$-valued models on a complete Boolean algebra $\mathsf{B}$, the mixing $\mathsf{B}$-valued models form a category which is dually equivalent to the category of sheaves on $\mathsf{B}$ (for the dense Grothendieck topology on the positive elements of $\mathsf{B}$, or -equivalently- for the sup-topology on $\mathsf{B}$ \cite[Examples (d), (e), pp. 114-115]{MacLMoer}), see \cite{PIEVIA19,PIEVIA-SHBOOVALMOD25}.

Let us continue our rough argumentation on why Boolean valued semantics is appropriate for infinitary logics.

In the sequel Tarski models denote for us the usual two valued structures, and we say Boolean/Tarski consistent to assert that a formula/theory has a Boolean valued model/Tarski model.
We now want to bring forward why the so called \emph{full} or \emph{mixing} models are the interesting Boolean valued models.

Given a $\mathsf{B}$-valued $\mathrm{L}$-model $\mathfrak{M}$ and an ultrafilter $F$ on
$\mathsf{B}$, the Tarski quotient $\mathfrak{M}_{/F}$ has as domain the set of equivalence classes for the relation 
	\[
		x \sim_F y \ \Leftrightarrow \ \Qp{x = y}_\mathsf{B}^\mathfrak{M} \in F.
	\]
	For atomic formulae, one sets that 
	
	\begin{equation}\label{eqn:fullness}
	\mathfrak{M}_{/F} \vDash \phi([x_1]_F,\ldots,[x_n]_F)\qquad \text{if and only if} \qquad \Qp{\phi(x_1,\ldots,x_n)}_\mathsf{B}^\mathfrak{M} \in F.
	\end{equation}
	It can be checked that $\mathfrak{M}_{/F}$ is a Tarski model in the signature $\mathrm{L}$, see Def. \ref{def:quotientmodel} for details.

One would like to have that  (\ref{eqn:fullness}) holds for all formulae rather than just for atomic formulae. This is not possible in general for $\mathrm{L}_{\infty\infty}$ and for all Boolean valued models (see for a counterexample \cite[Example 6.3.5]{VialeForcing}), but holds for the large subclass given by the full (or mixing) models relative to the fragment given by $\mathrm{L}_{\infty\omega}$ (see Def. \ref{def:fullmodel}, Proposition \ref{prop:mixfull}).

The following are nice consequences of admitting a full (mixing) model for an $\mathrm{L}_{\infty\omega}$-theory:
\begin{enumerate}
\item \label{observation1}
Assume $T$ is a \emph{countable} $\mathrm{L}_{\omega_1\omega}$-theory admitting 
a mixing $\mathsf{B}$-model $\mathfrak{M}$. Then $\mathfrak{M}_{/F}$ is a Tarski model of $T$ for a dense $G_\delta$-set of ultrafilters $F\in\mathsf{St}(\mathsf{B})$ (this is just a rephrasing in our terminology of Makkai's model existence theorem \cite[Thm. 2]{KeislerInfLog}).
\item \label{observation2}
Assume $T$ is a \emph{first order} theory admitting a mixing $\mathsf{B}$-model $\mathfrak{M}$. Then $\mathfrak{M}_{/F}$ is a Tarski model of $T$ for all ultrafilters $F\in\mathsf{St}(\mathsf{B})$.\footnote{See \cite[Section 3.8, Thm. 6.3.7]{VialeForcing} for the definition of $\mathsf{St}(\mathsf{B})$ and the proof of the mentioned theorem.}
\end{enumerate}
Theorem 3.1 in \cite{MatteoJuanBoolean} shows that any provably consistent $\mathrm{L}_{\infty\omega}$-theory admits a mixing model. Furthermore, it can be shown that any such theory existing in the set theoretic universe $V$ has a Tarski model in a generic extension of $V$. This has to be contrasted with a very nice (and as yet unpublished) result of Ben De Bondt stating that 
there is an $\mathrm{L}_{\omega_1\omega_1}$-sentence $\psi$ which is Boolean consistent, but has no  Tarski model in any generic extension.

In conclusion $\mathrm{L}_{\infty\omega}$ has an absolute Boolean valued semantics\footnote{I.e. for $\mathsf{B}$ a Boolean algebra, $\psi$ an $\mathrm{L}_{\infty\omega}$-sentence, $\mathfrak{M}$ a $\mathsf{B}$-valued $\mathrm{L}$-structure all existing in $V$,  the assertion $b=\Qp{\phi}_\mathsf{B}^\mathfrak{M}$ holds in $V$ if and only if it holds in some transitive outer model $W$ of $V$ (provided one pays a bit of attention in revising the definition of $\Qp{\phi}_\mathsf{B}^\mathfrak{M}$ to cover the case when $\mathsf{B}$ is not complete, so to accomodate the fact that $\mathsf{B}$ might not be a complete Boolean algebra in $W$).} which is complete for a natural proof system and is closely related to Tarski semantics (e.g. an $\mathrm{L}_{\infty\omega}$-theory is Boolean consistent if and only if it has a Tarski model in a forcing extension), on the other hand De Bondt's counterexample shows that the Boolean valued semantics for $\mathrm{L}_{\infty\infty}$ is complete for a natural proof system, but very far apart from Tarski semantics. Furthermore, even its restriction to Tarski models is not absolute.
We refer the reader also to \cite{SantiagoPhd,MatteoJuanBoolean} for a deeper analysis of the merits of our approach in the study of
$\mathrm{L}_{\infty\infty}$ and $\mathrm{L}_{\infty\omega}$.
In particular, \cite{SantiagoPhd} uses Thm. \ref{thm:boolcomp0} to produce an alternative proof of the celebrated result of Asperò and Schindler  \cite{ASPSCH(*)} that $\MM^{++}$ implies Woodin's axiom $(*)$. We are very optimistic that there will be other useful set theoretic applications of Thm. \ref{thm:boolcomp0} to produce iteration theorems for forcing notions. We also expect that Thm. \ref{thm:boolcomp0} can bear interesting fruits outside of set theory.


\section{Preliminaries}\label{sec:preliminaries}

In this section we state all the required background. Those accustomed to the use of infinitary logics and Boolean valued models can safely skip to Section \ref{sec:compactness}. For a more detailed account on these preliminaries see \cite{MatteoJuanBoolean,KeislerInfLog,ModelsGames,VialeForcing}.

\begin{notation}
A \emph{signature} $\mathrm{L}=\bp{R_i:i\in I}\cup\bp{f_j:j\in J}\cup\bp{c_k:k\in K}$ is a set of relation symbols $R_i$ of arity $\kappa_i$, function symbols $f_j$ of arity $\theta_j$ and constant symbols $c_k$. 

A \emph{first order signature} is a signature in which $\kappa_i$ and $\theta_j$ are finite for all $i,j$. 

A \emph{relational signature} is a signature with no function symbol.
\end{notation}

We restrict in most cases our attention to first order relational signatures, and our main results are stated and proved only for those signatures; with minor adjustments everything can be generalized to arbitrary first order signatures. We do this mainly to simplify our exposition and avoid some technicalities which would burden the proofs without shedding more light on the main ideas. On the other hand it is not transparent to us whether there can be natural generalizations of most of our results to arbitrary (non first order) signatures.

\subsection{Infinitary logics}

\begin{definition}
	Fix two cardinals $\lambda, \kappa$, a set of $\kappa$ variables $\{v_\alpha : \alpha < \kappa\}$, and consider a signature $\mathrm{L}$ whose function and relation symbols have all arity below $\kappa$. The set of terms and atomic formulae for $\mathrm{L}_{\kappa \lambda}$ is constructed as expected in analogy to first order logic using the symbols of $\mathrm{L}	\cup \{v_\alpha : \alpha < \kappa\}$. The other $\mathrm{L}_{\kappa \lambda}$-formulae are defined by induction as follows:

	\begin{itemize}
		\item if $\phi$ is an $\mathrm{L}_{\kappa \lambda}$-formula, then so is $\neg \phi$;
		\item if $\Phi$ is a set of $\mathrm{L}_{\kappa \lambda}$-formulae of size $< \kappa$ with free variables in the set $V=\bp{v_i:i\in I}$ for some $I\in [\kappa]^{<\lambda}$, then so are $\bigwedge \Phi$ and $\bigvee\Phi$;
		\item if $V=\bp{v_i:i\in I}$ for some $I\in [\kappa]^{<\lambda}$ and $\phi$ is an $\mathrm{L}_{\kappa \lambda}$-formula, then so are $\forall V \phi$ and $\exists V \phi$.
	\end{itemize}

	We let $\mathrm{L}_{\infty \lambda}$ be the family of $\mathrm{L}_{\kappa \lambda}$-formulae for some $\kappa$, and $\mathrm{L}_{\infty \infty}$ be the family of  $\mathrm{L}_{\kappa \lambda}$-formulae for some $\kappa,\lambda$.
\end{definition}

\subsection{Proof systems for $\Linff$}\label{subsec:proofsystem}

Any usual proof system for first order logic can be naturally generalized to $\mathrm{L}_{\infty\infty}$. For example a Gentzen type set of rules for $\Linff$ is the following:

\begin{displaymath}
\begin{array}{lc}
  
        	         \\
    \mbox{Equality Axiom 1} &
    \prftree{}{\vdash c=c }\\

\\
    \mbox{Equality Axiom 2} &
    \prftree{}{c = d \vdash d = c} \\

\\
   \mbox{Equality Axiom 3} &
    \prftree{}{ c = d, d=e \vdash c = e}\\

\\

 \mbox{Equality Axiom 4}&
    \prftree{}{\{u_i = t_i:i\in I\}, \phi(t_i:i\in I) \vdash \phi(u_i / t_i: i\in I)} \\
 \\
 \end{array}
\end{displaymath}

\begin{displaymath}
\begin{array}{lccl}
    \mbox{Axiom rule} &
    \prftree{}{\Gamma, \phi \vdash \phi, \Delta} &
	\prftree{\Gamma, \phi \vdash \Delta}{\Gamma' \vdash \phi, \Delta'}{\Gamma,\Gamma' \vdash \Delta, \Delta'} &   
    \mbox{Cut Rule} \\
    \mbox{Substitution} &
    \prftree{\Gamma \vdash \Delta}{\Gamma(\overline{w} / \overline{v}) \vdash \Delta(\overline{w} / \overline{v})} & 
    \prftree{\Gamma\vdash \Delta}{\Gamma,\Gamma' \vdash \Delta, \Delta'} &   
    \mbox{Weakening}
 \end{array}
\end{displaymath}

\begin{displaymath}
  \begin{array}{lccl}
    \mbox{Left-$\bigwedge$} &
    \prftree{\Gamma,\Gamma' \vdash \Delta}{\Gamma,\bigwedge \Gamma' \vdash \Delta} &
    \prftree{\Gamma \vdash \phi_i, \Delta \ ,\ i \in I}{\Gamma \vdash \bigwedge_{i \in I} \{\phi_i : i \in I\}, \Delta} &
    \mbox{Right-$\bigwedge$} \\
\\
   \mbox{Left-$\bigvee$} &
    \prftree{\Gamma\vdash \Delta,\Delta' }{\Gamma \vdash \Delta,\bigvee \Delta'} &
    \prftree{\Gamma,\phi_i \vdash , \Delta \ ,\ i \in I}{\Gamma,\bigvee_{i \in I} \{\phi_i : i \in I\} \vdash  \Delta} &
    \mbox{Right-$\bigvee$} \\
  \end{array}
\end{displaymath}

\begin{displaymath}
  \begin{array}{lccl}
    \mbox{Left-$\forall$} &
    \prftree{\Gamma, \phi(\overline{t} /\overline{v}) \vdash \Delta}{\Gamma, \forall \overline{v} \phi(\overline{v}) \vdash \Delta} &
    \prftree[r]{(*)}{\Gamma \vdash \phi(\overline{v} ), \Delta}{\Gamma \vdash \forall \overline{v}\phi(\overline{v}), \Delta} &
    \mbox{Right-$\forall$}\\
    						   \\

  \mbox{Left-$\exists$} &
    \prftree[l]{(*)}{\Gamma, \phi(\overline{v} )\vdash \Delta}{\Gamma, \exists\overline{v} \phi(\overline{v}) \vdash \Delta} &
    \prftree{\Gamma \vdash \phi(\overline{t} / \overline{v}) , \Delta}{\Gamma \vdash \exists \overline{v}\phi(\overline{v}), \Delta} &
    \mbox{Right-$\exists$}\\
    						   \\

  \end{array}
\end{displaymath}

$(*)$ None of the variables appearing in $\overline{v}$ can occur free in $\Gamma\cup\Delta$.
%

A proof of $\Gamma\vdash\Delta$ is a well founded labelled tree whose root is $\Gamma\vdash\Delta$, whose leaves are axioms,  and whose other nodes are obtained from
their immediate successors applying one of the above rules.

In this context we shall (somewhat loosely, but conforming to standard practice) write $\Gamma\vdash\Delta$ to intend that $\Delta$ is derivable from $\Gamma$, i.e. -more correctly- \emph{there is a proof of the string $\Gamma\vdash\Delta$}, and write $\phi \vdash \psi$ when $\{\psi\}$ is derivable from $\{\phi\}$.


\subsection{Boolean valued models} \label{subsec:boolvalsem}

To simplify our exposition we limit ourselves to consider models for relational first order signatures (i.e. signatures without function symbols, and with all their relation symbols of finite arity). With little effort the definitions and results of the present paper can be generalized to arbitrary first order signatures, while it is not clear to us what might be their correct generalization to signatures with relation and function symbols of infinite arity.
We refer the reader to \cite[Chapter 6]{VialeForcing} for details on the content of this section.

\begin{definition} \label{def:Bvalmod}
Let $\mathrm{L}$ be a relational first order signature and $\mathsf{B}$ be a complete Boolean algebra. A $\mathsf{B}$-valued model $\mathcal{M}$ for $\mathrm{L}$ is given by:
	\begin{enumerate}
		\item a non-empty set $M$;
		\item the Boolean value of equality,
		\begin{align*}
			M^2 &\rao \mathsf{B} \\
			(\tau,\sigma) &\mapsto \Qp{\tau=\sigma}^{\mathcal{M}}_\mathsf{B};
		\end{align*} 
		\item the interpretation of relation symbols $R \in \mathrm{L}$ of arity $n$ by maps
			\begin{align*}
				M^n &\rao \mathsf{B} \\
				(\tau_i:i<n) &\mapsto \Qp{R (\tau_i:i<n) }^{\mathcal{M}}_\mathsf{B};
			\end{align*}
		\item the interpretation $c^\mathcal{M} \in M$ of constant symbols $c$ in $\mathrm{L}$.
	\end{enumerate}

	We require that the following conditions hold:

	\begin{enumerate}[(A)]
		\item For all $\tau,\sigma,\pi \in M$,
			\begin{gather*}
				\Qp{\tau=\tau}^{\mathcal{M}}_\mathsf{B} = 1_\mathsf{B}, \\
				\Qp{\tau=\sigma}^{\mathcal{M}}_\mathsf{B} = \Qp{\sigma=\tau}^{\mathcal{M}}_\mathsf{B}, \\
				\Qp{\tau=\sigma}^{\mathcal{M}}_\mathsf{B} \wedge \Qp{\sigma=\pi}^{\mathcal{M}}_\mathsf{B} \leq \Qp{\tau=\pi}_\mathsf{B}^{\mathcal{M}}.
			\end{gather*}
		\item \label{eqn:subslambda} If $R \in \mathrm{L}$ is an $n$-ary relation symbol, for all $(\tau_i:\,i < n), (\sigma_i:\,i < n) \in M^n$,
			\begin{equation*}
				\bigg(\bigwedge_{i < n}\Qp{\tau_i=\sigma_i}^{\mathcal{M}}_\mathsf{B} \bigg) \wedge \Qp{R(\tau_i:\,i < n)}^{\mathcal{M}}_\mathsf{B} \leq 					\Qp{R(\sigma_i:\,i< n)}^{\mathcal{M}}_\mathsf{B}.
			\end{equation*}
	\end{enumerate}
\end{definition}

%
%

\begin{definition}\label{def:boolvalsem} 
	Fix $\mathsf{B}$ a Boolean algebra and $\mathcal{M}$ a $\mathsf{B}$-valued structure for a signature $\mathrm{L}$. We define the $\RO(\mathsf{B}^+)$-value\footnote{When $\bool{B}$ is not a complete boolean algebra, we let $\RO(\bool{B}^+)$ denote its boolean completion (according to the terminology of \cite[Ch. 4]{VialeForcing}). W.l.o.g. the reader may assume throughout the paper that the boolean algebras $\bool{B}$ under consideration are always complete, hence isomorphic to $\RO(\bool{B}^+)$.} of an $\mathrm{L}_{\infty\infty}$-formula $\phi(\overline{v})$ with assignment $\nu:\overline{v} \mapsto \overline{m}$ by induction as follows:

	\begin{gather*}
		\Qp{R(t_i:i < n)[\overline{v} \mapsto \overline{m}]}^{\mathcal{M},\nu}_{\RO(\bool{B}^+)} = \Qp{R(t_i[\overline{v} \mapsto \overline{m}]:i < n)}^{\mathcal{M},\nu}_\mathsf{B} \text{ for $R\in\mathrm{L}$ of arity $n$},\\
	\Qp{(\neg \phi)[\overline{v} \mapsto \overline{m}]}^{\mathcal{M},\nu}_{\RO(\bool{B}^+)}  = \neg \Qp{\phi[\overline{v} \mapsto \overline{m}]}^{\mathcal{M},\nu}_{\RO(\bool{B}^+)} ,\\
		\Qp{(\bigwedge \Phi)[\overline{v} \mapsto \overline{m}]}^{\mathcal{M},\nu}_{\RO(\bool{B}^+)}  = \bigwedge_{\phi \in \Phi} \Qp{\phi[\overline{v} \mapsto \overline{m}]}^{\mathcal{M},\nu}_{\RO(\bool{B}^+)} ,\\
		\Qp{(\bigvee \Phi)[\overline{v} \mapsto \overline{m}]}^{\mathcal{M},\nu}_{\RO(\bool{B}^+)}  = \bigvee_{\phi \in \Phi} \Qp{\phi[\overline{v} \mapsto \overline{m}]}^{\mathcal{M},\nu}_{\RO(\bool{B}^+)} ,\\
		\Qp{(\forall V \phi)[\overline{v} \mapsto \overline{m}]}^{\mathcal{M},\nu}_{\RO(\bool{B}^+)}  = \bigwedge_{\overline{a} \in M^V} \Qp{\phi[\overline{v} \mapsto \overline{m}, V \mapsto \overline{a}]}^{\mathcal{M},\nu}_{\RO(\bool{B}^+)} ,\\ 
		\Qp{(\exists V \phi)[\overline{v} \mapsto \overline{m}]}^{\mathcal{M},\nu}_{\RO(\bool{B}^+)}  = \bigvee_{\overline{a} \in M^V} \Qp{\phi[\overline{v} \mapsto \overline{m}, V \mapsto \overline{a}]}^{\mathcal{M},\nu}_{\RO(\bool{B}^+)}.
	\end{gather*}
\end{definition}

It can be checked that the Boolean truth values in a Boolean valued model given to $\Linff$-sentences is independent of the assignment to free variables. We usually omit as many superscripts and subscripts as possible in the evaluation map $\phi\mapsto\Qp{\phi}$, when those are clear from the context. Also we may write at times $\Qp{\phi(c)}$ ($\Qp{\phi(c_i:i\in I)}$) and 
$\Qp{\phi(c^\mathcal{M})}$
($\Qp{\phi(c_i^{\mathcal{M}}:i\in I)}$): the first is a sloppy notation for the evaluation in the model $\mathcal{M}$ of the sentence obtained by the substitution of the free variable $x$ (the string of free variables $(x_i:i\in I)$) by $c$ (by $(c_i:i\in I)$) in $\phi(x)$ ($\phi(x_i:i\in I)$). The second is a shorthand for what above has been denoted by $\Qp{\phi(x\mapsto c^{\mathcal{M}})}$ (by $\Qp{\phi(x_i\mapsto c_i^{\mathcal{M}}:i \in I)}$). Both notations reproduce a common practice among logicians.

The quotient of  a $\bool{B}$-valued model by a filter $F$ on $\bool{B}$ instantiates in this context the algebraic notion of quotient by an ideal:

\begin{definition}\label{def:quotientmodel}
 Let $\mathsf{B}$ be a Boolean algebra and let $\mathrm{L}=\bp{R_i:i\in I,c_j:j\in J}$ be a relational
language where $R_i$ is a $m_i$-ary relation symbol for every $i\in I$ and each $c_j$ is a constant symbol.
Let also $\mathcal{M}=(M,R_i:i\in I,c_j:j\in J)$ be a $\mathsf{B}$-model for $\mathrm{L}$. 
Let $F$ be a filter on $\bool{B}$. The \textit{$F$-quotient} $\mathcal{M}/_F=(M/_F,R_i/_F: i\in I, [c_j^\mathcal{M}]_F: j\in J)$ is the $\bool{B}/_F$-model defined as follows:
\begin{itemize}
\item [-] $[h]_F=\bp{f\in M:\Qp{f=h}\in F}$ for $h\in M$;
 \item [-] $M/_F=\bp{[h]_F:h\in M}$;
\item [-] $\Qp{R_i/_F([f_1]_F,\dots,[f_{m_i}]_F)}^{\mathcal{M}/_F}_{\bool{B}/_F}=[\Qp{R_i(f_1,\dots,f_n)}^{\mathcal{M}}_{\bool{B}}]_F$
for every $i\in I$.
\end{itemize}

\end{definition}

It can be checked that the $F$-quotient of a $\mathsf{B}$-valued model for a first order signature is a well defined $\bool{B}/_F$-valued model;
 hence it is a Tarski model when $F$ is an ultrafilter.

\begin{definition} \label{def:mixingmodel}
Let $\bool{B}$ be a complete Boolean algebra and $\mathcal{M}$ be a $\mathsf{B}$-valued model for some signature $\mathrm{L}$. Given a cardinal $\lambda$, 
$\mathcal{M}$ has the \emph{$\lambda$-mixing property} if for any antichain $A \subset \mathsf{B}$ of size at most $\lambda$ and $\{\tau_a : a \in A\} \subset M$ there is some $\tau \in M$ such that $a \leq \Qp{\tau=\tau_a}_\mathsf{B}$ for all $a \in A$.

$\mathcal{M}$ has the \emph{mixing property} if it has the $|\bool{B}|$-mixing property.
\end{definition}

\begin{definition} \label{def:fullmodel}
	Let $\lambda \leq \kappa$ be infinite cardinals, $\bool{B}$ be a complete Boolean algebra, and $\mathcal{M}$ be a $\mathsf{B}$-valued model for $\mathrm{L}_{\kappa\lambda}$. $\mathcal{M}$ is \emph{full} for the logic $\mathrm{L}_{\kappa\lambda}$ if for every $\mathrm{L}_{\kappa\lambda}$-formula $\phi(\overline{v},\overline{w})$ and $\overline{m} \in M^{\overline{w}}$ there exists $\overline{n} \in M^{\overline{v}}$ such that

	\[ 
		\Qp{\exists \overline{v} \phi(\overline{v},\overline{m})}_\mathsf{B} = \Qp{\phi(\overline{n},\overline{m})}_\mathsf{B}.
	\]
\end{definition}

\begin{proposition}\label{prop:mixfull}
	Let $\mathrm{L}$ be a signature and $\bool{B}$ a complete Boolean algebra. 
	
	\begin{itemize}
		\item Any $\mathsf{B}$-valued model for $\mathrm{L}$ with the mixing property is full for $\mathrm{L}_{\infty \omega}$.
		\item Assume the $\bool{B}$-valued model $\mathcal{M}$ (with domain $M$)
		
		\begin{itemize}
			\item is full for $\mathrm{L}_{\infty \omega}$,
			\item has the $2$-mixing property, and
			\item for some $\tau,\sigma$ in $M$, $\Qp{\tau = \sigma}^{\mathcal{M}}_\mathsf{B} = 0$.
		\end{itemize} 
		
		Then $\mathcal{M}$ has the mixing property.
	\end{itemize}
\end{proposition}

A proof of the first item can be obtained with minor twists from that of \cite[Proposition 2.9]{MatteoJuanBoolean}.

\begin{proof}
	Let $A$ be an antichain and $\{\tau_a : a \in A\}$ be a subset of $\mathcal{M}$. We need to find $\rho\in M$ such that 
	
	\[
		a \leq \Qp{\rho = \tau_a}_\mathsf{B} \hspace{0,5cm} \forall a \in A.
	\]
	\vspace{-0,1cm}
	
	\noindent By hypothesis there exist $\tau,\sigma$ such that $\Qp{\tau = \sigma}_\mathsf{B} = 0$. By $2$-mixing for $\mathcal{M}$, for each antichain $\{a,\neg a\}$ we can find $\sigma_a$ such that 
	
	\[
		\Qp{\sigma_a = \tau}_\mathsf{B} = a \ \ \wedge \ \  \Qp{\sigma_a = \sigma}_\mathsf{B} = \neg a.
	\]
	\vspace{-0,1cm}
	
	\noindent Consider the formula 
	
	\[
		\exists x \bigvee_{a \in A} \big( x = \tau_a \wedge \sigma_a = \tau \big).
	\]
	\vspace{-0,1cm}
	
	\noindent Let us check that its truth value is $\bigvee A$. 

	\begin{itemize}
		\item First, since $\Qp{\sigma_a = \tau}_\mathsf{B} = a$, we have that 
		
		\[
			\Qp{\rho = \tau_a}_\mathsf{B} \wedge \Qp{\sigma_a = \tau}_\mathsf{B} \leq a
		\]
		\vspace{-0,1cm}

		\noindent for all $\rho \in M$: therefore 
		
		\[
			\Qp{\exists x \bigvee_{a \in A} \big( x = \tau_a \wedge \sigma_a = \tau \big)}_\mathsf{B} \leq  \bigvee A.
		\]
		\vspace{-0,1cm}
		
		\item Secondly:
		
		\begin{align*}
			\Qp{\exists x \bigvee_{a \in A} \big( x = \tau_a \wedge \sigma_a = \tau \big)}_\mathsf{B}\geq\Qp{ \bigvee_{a \in A} \big( \tau_a = \tau_a \wedge \sigma_a = \tau \big)}_\mathsf{B} = \\
=\bigvee_{a \in A} (\Qp{\tau_a = \tau_a}_\mathsf{B} \wedge \Qp{\sigma_a = \tau}_\mathsf{B}) = \bigvee A.
		\end{align*}

	\end{itemize}	
	
	By fullness of $\mathcal{M}$, we can find $\rho$ such that 
	
	\begin{align*}
		\bigvee A = \Qp{\exists x \bigvee_{a \in A} \big( x = \tau_a \wedge \sigma_a = \tau \big)}_\mathsf{B}=\Qp{\bigvee_{a \in A} \big( \rho = \tau_a \wedge \sigma_a = \tau \big)}_\mathsf{B} = \\
=\bigvee_{a \in A} (\Qp{\rho = \tau_a}_\mathsf{B} \wedge \Qp{\sigma_a = \tau}) = \bigvee_{a \in A} \Qp{\rho = \tau_a}_\mathsf{B} \wedge a.
	\end{align*}

	\noindent Since $A$ is an antichain, we must have that $\Qp{\rho = \tau_a}_\mathsf{B} \wedge a = a$ for all $a\in A$. We conclude $a \leq \Qp{\rho = \tau_a}_\mathsf{B}$ holds for all $a\in A$, as was to be shown.
\end{proof}

\subsection{Semantics and notions of logical consequence for $\mathrm{L}_{\infty \omega}$ and $\mathrm{L}_{\infty \infty}$}

We recall some other key definitions and results from \cite{MatteoJuanBoolean}.
\begin{definition} \label{def:BolCon}

	$\mathrm{BVM}$ denotes the class of Boolean valued models, and $\mathrm{Sh}$ the subclass of Boolean valued models which have the mixing property. Let $\Gamma$ and $\Delta$ be sets of $\mathrm{L}_{\infty\infty}$-formulae. In case $\Gamma = \emptyset$ we let  
	
	\[
		\Qp{\bigwedge \Gamma}_\bool{B}^{\mathcal{M},\nu} = 1_{\bool{B}},
	\]
	and if $\Delta = \emptyset$ we let
	\[
		\Qp{\bigvee \Delta}_\bool{B}^{\mathcal{M},\nu} = 0_\bool{B}
	\]
for any $\bool{B}$-valued model $\mathcal{M}$ and assignement $\nu$.
	
	\begin{itemize}
		\item $\Gamma$ is \emph{Boolean satisfiable} (or Boolean consistent) if there is a complete Boolean algebra $\bool{B}$, a $\bool{B}$-valued model $\mathcal{M}$, and an assignemnt $\nu$ such that $\Qp{\phi}^{\mathcal{M},\nu}_\bool{B}=1_\bool{B}$ for each $\phi\in \Gamma$.
		\item $\Gamma\vDash_\mathrm{BVM} \Delta$ if 
			\[
				\Qp{\bigwedge\Gamma}^{\mathcal{M},\nu}_\bool{B}\leq\Qp{\bigvee\Delta}^{\mathcal{M},\nu}_\bool{B}
			\]
		for any complete Boolean algebra $\bool{B}$,
		$\bool{B}$-valued model $\mathcal{M}$, and assignement $\nu$.
		\item $\Gamma\vDash_\mathrm{Sh} \Delta$ if 
			\[
				\Qp{\bigwedge\Gamma}^{\mathcal{M},\nu}_\bool{B}\leq\Qp{\bigvee\Delta}^{\mathcal{M},\nu}_\bool{B}
			\]
		for any complete Boolean algebra $\bool{B}$, 
		$\bool{B}$-valued model $\mathcal{M}$ \emph{with the mixing property}, and assignment $\nu$.
	
		\item $\psi \equiv_{\mathrm{BVM}} \phi$ (respectively $\psi \equiv_{\mathrm{Sh}} \phi$) if $\bp{\psi} \vDash_\mathrm{BVM} \bp{\phi}$ and $\bp{\phi} \vDash_\mathrm{BVM}\bp{\psi} $ (respectively $\bp{\psi} \vDash_\mathrm{Sh} \bp{\phi}$ and $\bp{\phi} \vDash_\mathrm{Sh}\bp{\psi} $).
	\end{itemize}
\end{definition}

\begin{theorem}[Boolean Completeness for $\mathrm{L}_{\infty\infty}$]  \label{them:boolcompl}
	The following are equivalent for $T,S$ sets of $\mathrm{L}_{\infty\infty}$-formulae.
	\begin{enumerate}
		\item \label{thm:boolcomp2} $T\models_{\mathrm{BVM}}S$,
		\item \label{thm:boolcomp3}	$T\vdash S$.
	\end{enumerate}
	
	Furthermore, if $T,S$ are sets of $\mathrm{L}_{\infty \omega}$-formulae, one also has the equivalence of any of the above items with 
	
	\begin{itemize}
		\item[3.] \label{thm:boolcomp1} $T\models_{\mathrm{Sh}}S$.
	\end{itemize}
	
\end{theorem}

We note that the $\Linff$ version of the above completeness theorem (i.e. the equivalence between \ref{thm:boolcomp2} and \ref{thm:boolcomp3}  of the theorem) is due to Mansfield \cite{MansfieldConPro}. The proof of the above equivalences for $\mathrm{L}_{\infty \omega}$ appears in \cite{MatteoJuanBoolean}.
In \cite[Theorems 3.2, 3.8, 3.10]{MatteoJuanBoolean} we also proved natural generalizations to Boolean valued semantics for $\Linff$ of the interpolation and omitting types theorems holding for first order logic, we refer the reader to that paper for details.

\subsection{Consistency properties}

We adopt the following conventions:
\begin{notation}
The languages $\mathrm{L}$ under consideration are relational and have as set of constants $\mathcal{D}$.

An $\Linff$-sentence is a formula without free variables, and an $\Linff$-theory is a set of sentences.
\end{notation}

\begin{definition} \label{MovNegIns} Let $\phi$ be an $\mathrm{L}_{\infty \infty}$-formula. 
We define $\phi \neg$ (moving a negation inside) by induction on the complexity of formulae:

\begin{itemize}
\item If $\phi$ is an atomic formula $\varphi$, $\phi \neg$ is $\neg \varphi$.
\item If $\phi$ is $\neg \varphi$, $\phi \neg$ is $\varphi$.
\item If $\phi$ is $\bigwedge \Phi$, $\phi \neg$ is $\bigvee \{\neg \varphi : \varphi \in \Phi\}$.
\item If $\phi$ is $\bigvee \Phi$, $\phi \neg$ is $\bigwedge \{\neg \varphi : \varphi \in \Phi\}$.
\item If $\phi$ is $\forall \vec{v} \varphi(\vec{v})$, $\phi \neg$ is $\exists \vec{v} \neg \varphi (\vec{v})$.
\item If $\phi$ is $\exists \vec{v} \varphi(\vec{v})$, $\phi \neg$ is $\forall \vec{v} \neg \varphi (\vec{v})$.
\end{itemize}
\end{definition}

\begin{definition} \label{def:ConProInf}
	Let $\mathrm{L} = \mathcal{R} \cup \mathcal{D}$ be a signature where the relation symbols are in $\mathcal{R}$ and $\mathcal{D}$ is the set of constants. Given an infinite set of constants $\mathcal{C}$ disjoint from $\mathcal{D}$, consider $\mathrm{L}(\mathcal{C})$ the signature obtained by extending $\mathrm{L}$ with the constants in $\mathcal{C}$. A set $S$ whose elements are set sized subsets of $\mathrm{L}(\mathcal{C})_{\infty\infty}$ is a consistency property for $\mathrm{L}(\mathcal{C})_{\infty\infty}$ if for each $s \in S$ the following properties hold:

	\begin{enumerate}
		\item[(Con)]\label{conspropCon} for any $\mathrm{L}(\mathcal{C})_{\infty\infty}$-sentence $\phi$ either $\phi\not\in s$ or $\neg\phi\not\in s$,
		\item[(Ind.1)]\label{conspropInd1} if $\neg \phi \in s$, $s \cup \{\phi \neg\} \in S$,
		\item[(Ind.2)]\label{conspropInd2} if $\bigwedge \Phi \in s$, then for any $\phi \in \Phi$, $s \cup \{\phi\} \in S$,
		\item[(Ind.3)]\label{conspropInd3} if $\forall \vec{v} \phi(\vec{v}) \in s$, then for any $\vec{c} \in (\mathcal{C}\cup\mathcal{D})^{|\vec{v}|}$, $s \cup \{\phi(\vec{c})\} \in S$,
		\item[(Ind.4)]\label{def:conspropInd4} if $\bigvee \Phi \in s$, then for some $\phi \in \Phi$, $s \cup \{\phi\} \in S$,
		\item[(Ind.5)]\label{def:conspropInd5} if $\exists \vec{v} \phi(\vec{v}) \in s$, then for some $\vec{c} \in \mathcal{C}^{|\vec{v}|}$, $s \cup \{\phi(\vec{c})\} \in S$,
		\item[(Str.1)]\label{def:conspropStr1} if $c,d \in \mathcal{C}\cup\mathcal{D}$ and $c = d \in s$, then $s \cup \{d = c\} \in S$,
		\item[(Str.2)] \label{def:conspropStr2} if $c,d \in \mathcal{C} \cup \mathcal{D}$ and $\{c = d, \phi(d)\} \subset s$, then $s \cup \{\phi(c)\} \in S$,
		\item[(Str.3)] \label{def:conspropStr3} if $d \in \mathcal{C} \cup \mathcal{D}$, then for some $c \in \mathcal{C}$, $s \cup \{c = d\} \in S$.
	\end{enumerate} 
\end{definition}

The following theorem is essentially due to Mansfield for its first part \cite{MansfieldConPro}, while the stronger conclusion for consistency properties consisting of $\Linf$-sentences is \cite[Thm. 5.15]{MatteoJuanBoolean}:

\begin{theorem}[Model Existence Theorem for $\Linff$] \label{ModExiThe} 
	Let $\mathrm{L}$ be a language, $\mathcal{C}$ be a set of fresh constants and $S$ be a consistency property consisting of $\mathrm{L}(\mathcal{C})_{\infty \infty}$-sentences. Then every $s \in S$ is Boolean consistent (and holds in a model with domain $\mathcal{C}$). Furthermore, if $S$ only contains $\Linf$-formulae, for each $s\in S$ there is a witness of its Boolean consistency with the mixing property.	
\end{theorem}

\subsection{On the quantifier free version of an $\Linff$-theory}\label{subsec:qfversionofT}

We show that the problem of Boolean consistency for $\Linff$-theories can be reduced to 
the simpler case in which the theories under examination consist uniquely of quantifier free $\Linff$-sentences.

\begin{definition}\label{def:QECaxiom}
Given a set of constants $\mathcal{C}$, the quantifier elimination axiom
$\bool{QE}_{\mathcal{C}}$ is the sentence:
\[
\forall x\,\bigvee_{c\in\mathcal{C}}(x=c).
\]

Given an $\Linff$-sentence $\phi$, $\phi_{\mathcal{C}}$ is the quantifier free sentence obtained by replacing along the tree of subformulae of $\phi$ the
quantified subformulae of type:
\begin{itemize}
\item
$\forall \vec{x}\,\phi(\vec{x})$ by $\bigwedge_{\vec{c}\in\mathcal{C}^{\vec{x}}}\phi(\vec{c})$,
\item
$\exists \vec{x}\,\phi(\vec{x})$ by $\bigvee_{\vec{c}\in\mathcal{C}^{\vec{x}}}\phi(\vec{c})$.
\end{itemize}

Given an $\Linff$-theory $T$, $T_{\mathcal{C}}$ is the $\Linff$-theory:
\[
\bp{\phi_{\mathcal{C}}: \phi\in T}\cup\bp{\bool{QE}_{\mathcal{C}}}.
\]
\end{definition}

For $\Linf$-theories one can argue that they are essentially equivalent (modulo the 
$\bool{QE}_{\mathcal{C}}$-axiom) to a quantifier free $\Linf$-theory. This follows from the following:
\begin{fact}
Let $\psi$ be an $\Linf$-sentence. Then for all infinite set of fresh constants $\mathcal{C}$ for $\mathrm{L}$
\[
\mathrm{QE}_{\mathcal{C}}\models_{\bool{BVM}}\psi\leftrightarrow \psi_{\mathcal{C}}.
\]
\end{fact}
\begin{proof}
Assume in a $\bool{B}$-model $\mathcal{M}$ with domain $M$ for $\mathrm{L}\cup\mathcal{C}=\mathrm{L}(\mathcal{C})$,
$b=\Qp{\mathrm{QE}_{\mathcal{C}}}>0$.
W.l.o.g. by considering $\mathcal{M}$ a $\bool{B}\restriction b$-model, we can assume $b=1$.
Then for all $m\in M$
\[
\bigvee_{c\in\mathcal{C}}\Qp{m=c^{\mathcal{M}}}=1.
\]
Hence, for a $\mathrm{L}(\mathcal{C})_{\infty\infty}$-formula $\phi(x)$ in \emph{unique} displayed free variable $x$ and all $m\in M$:
\begin{equation}\label{eqn:keyidQE}
\Qp{\phi(m)}=\Qp{\phi(m)}\wedge\bigvee_{c\in\mathcal{C}}\Qp{c^{\mathcal{M}}=m}=\bigvee_{c\in\mathcal{C}}(\Qp{\phi(m)}\wedge\Qp{c^{\mathcal{M}}=m})\leq
\bigvee_{c\in\mathcal{C}}\Qp{\phi(c^{\mathcal{M}})}.
\end{equation}
Therefore
\[
\Qp{\exists x\,\phi(x)}=\bigvee_{m\in M}\Qp{\phi(m)}\leq
\bigvee_{m\in M}\bigvee_{c\in\mathcal{C}}\Qp{\phi(c^{\mathcal{M}})}=\bigvee_{c\in\mathcal{C}}\Qp{\phi(c^{\mathcal{M}})}\leq \bigvee_{m\in M} \Qp{\phi(m)} = \Qp{\exists x\,\phi(x)}.
\]
By using the infinitary De Morgan laws for negation, we also get
\[
\Qp{\forall x\,\phi(x)}=
\Qp{\neg\exists x\,\neg\phi(x)}=\neg\bigvee_{c\in\mathcal{C}}\neg\Qp{\phi(c^{\mathcal{M}})}=
\bigwedge_{c\in\mathcal{C}}\Qp{\phi(c^{\mathcal{M}})}.
\]
Going back to the $\bool{B}$-valued model $\mathcal{M}$, we retain that
\[
\Qp{\mathrm{QE}_{\mathcal{C}}}\wedge \Qp{\exists x\,\phi(x)}=\Qp{\mathrm{QE}_{\mathcal{C}}} \wedge \bigvee_{c\in\mathcal{C}}\Qp{\phi(c^{\mathcal{M}})}
\]
and
\[
\Qp{\mathrm{QE}_{\mathcal{C}}}\wedge \Qp{\forall x\,\phi(x)}=\Qp{\mathrm{QE}_{\mathcal{C}}} \wedge \bigwedge_{c\in\mathcal{C}}\Qp{\phi(c^{\mathcal{M}})}
\]
holds for all $\bool{B}$-valued model $\mathcal{M}$ for $\mathrm{L}(\mathcal{C})$ and for all 
$\Linff$-formulae $\phi(x)$ in the \emph{unique} free variable $x$.

By systematically applying these identities, we obtain that for any $\Linf$-sentence $\psi$
in all Boolean valued models $\mathcal{M}$ for $\mathrm{L}(\mathcal{C})$
\[
\Qp{\mathrm{QE}_{\mathcal{C}}}\leq\Qp{\psi\leftrightarrow \psi_{\mathcal{C}}}.
\]
\end{proof}

We do not expect the Fact to hold for $\Linff$-sentences: the natural generalization of the above proof to an existential quantifier on an infinite string $(x_i:i\in I)$ would require to establish the identity 
\[
\Qp{\phi(m_i:i\in I)}=\bigvee_{\vec{c}\in\mathcal{C}^I}\Qp{\phi(c_i^{\mathcal{M}}:i\in I)}
\]
corresponding to the key identity (\ref{eqn:keyidQE}).
This may not hold in Boolean algebras $\bool{B}$ for which
the distributive law
\[
\bigwedge_{i\in I}\bigvee_{c\in\mathcal{C}}\Qp{m_i=c}=
\bigvee_{f\in \mathcal{C}^I}\bigwedge_{i\in I}\Qp{m_i=f(i)}
\]
fails.

Nonetheless, there is a cheap way around this obstacle for $\Linff$-theories, rooted in the Model Existence Theorem \ref{ModExiThe}:

%

\begin{fact}
Let $T$ be an $\Linff$ theory. TFAE:
\begin{itemize}
\item
$T$ is Boolean consistent;
\item
$T_{\mathcal{C}}$ is Boolean consistent for some fresh (for $\mathrm{L}$) set of constants $\mathcal{C}$.
\end{itemize} 
\end{fact}

\begin{proof}
If $\mathcal{M}$ is a model of $T$ with domain $M$, then $T_{M}$ is Boolean consistent, as it holds in $\mathcal{M}$, since the identities
\[
\Qp{\forall (x_i:i\in I)\,\phi(x_i:i\in I)}=\bigwedge_{f\in M^I}\Qp{\phi(f(i)^{\mathcal{M}}:i\in I)}=\bigwedge_{f\in M^I}\Qp{\phi(f(i):i\in I)}
\]
and  
\[
\Qp{\exists (x_i:i\in I)\,\phi(x_i:i\in I)}=\bigvee_{f\in M^I}\Qp{\phi(f(i)^{\mathcal{M}}:i\in I)}=\bigvee_{f\in M^I}\Qp{\phi(f(i):i\in I)}
\]
clearly
hold in $\mathcal{M}$. 

Conversely, if $T_{\mathcal{C}}$ is Boolean consistent,
\[
\bp{s: \,s \text{ is a finite set of $\mathrm{L}(\mathcal{C})$-sentences and }s\cup T_{\mathcal{C}}\text{ is Boolean consistent}}
\]
is a consistency property which (by the $\mathsf{QE}_{\mathcal{C}}$-axiom) satisfies Clause (Str.3) of Def. \ref{def:ConProInf} for the set of constants $\mathcal{C}$; hence, by Thm. \ref{ModExiThe}, it has a Boolean valued model with domain $\mathcal{C}$ which also validates $T_{\mathcal{C}}$.
In such a Boolean valued model the identities
\[
\Qp{\forall (x_i:i\in I)\,\phi(x_i:i\in I)}=\bigwedge_{f\in\mathcal{C}^I}\Qp{\phi(f(i)^{\mathcal{M}}:i\in I)}=\bigwedge_{f\in\mathcal{C}^I}\Qp{\phi(f(i):i\in I)}
\]
and  
\[
\Qp{\exists (x_i:i\in I)\,\phi(x_i:i\in I)}=\bigvee_{f\in\mathcal{C}^I}\Qp{\phi(f(i)^{\mathcal{M}}:i\in I)}=\bigvee_{f\in\mathcal{C}^I}\Qp{\phi(f(i):i\in I)}
\]
hold by definition. Hence, in this model, one has that $T$ holds as well.
\end{proof}
In particular, while in this paper we will not pay particular attention to the syntactic nature of the $\Linff$-theories under examination, when studying the problem of their Boolean consistency, one can appeal to the above facts to reduce it to the case of quantifier free $\Linff$-theories satisfying $\bool{QE}_{\mathcal{C}}$. This makes the task of showing their consistency somewhat simpler (at least notationally). It also explains why the notions of subsentence of an $\Linff$-sentence and of conservative strengthening to be introduced later in Section \ref{subsec:constrfincns} are natural.


\section{The Compactness Theorem for $\Linff$}\label{sec:compactness}

Suppose $\{\psi_i : i \in I\}$ is a set of first order sentences. Compactness for first order logic says that every finite subset of $\{\psi_i : i \in I\}$ has a Tarski model if and only if $\bigwedge_{i \in I} \psi_i$ has a Tarski model. With this formulation the result does not generalize to $\Linff$ (nor to $\Linf$), even if one replaces Tarski models with Boolean valued models (see Example \ref{faicom}).\medskip

We produce a generalization to $\Linff$ of the compactness theorem; toward this aim we introduce the key concept of \textbf{conservative strengthening} and the corresponding notion of being a \textbf{finitely conservative} set of $\Linff$-sentences. We will show in Section \ref{FinConBooConTarCon} that (with minor twists) for families of \emph{first order sentences} being finitely conservative  is a natural refinement of being finitely consistent (and in a precise sense an equivalent reformulation of this concept).
\medskip

Replacing \emph{finitely consistent} with \emph{finitely conservative} and \emph{Tarski} semantics with \emph{Booolean valued} semantics, compactness generalizes naturally to $\Linff$ (see Theorem \ref{conservative}).

\subsection{The failure of the simplistic notion of compactness for $\Linff$} 

The following example shows that the natural generalization of compactness fails already for the smallest non-trivial fragment of $\Linff$, i.e. for a countable $\mathrm{L}_{\omega_1\omega}$-theory.

\begin{example}[The failure of compactness] \label{faicom} Let $\mathcal{L}$ be a language containing constants $\{c_n : n \in \omega\} \cup \{c_{\omega}\}$ and the equality relation symbol. Denote by $T$ the theory
	\begin{gather*}
		\{ c_n \neq c_\omega : n \in \omega\} \ \cup \\
		\{\bigvee_{n \in \omega} c_\omega = c_n\}.
	\end{gather*}
	
	\ding{95} \textbf{Failure of compactness for Tarski semantics.} Let us argue that this theory has no Tarski model, yet it has models for every finite subset. $T$ has no Tarski model since any realization $c_m = c_\omega$ of the axiom 

	\[
		\bigvee_{n \in \omega} c_\omega = c_n
	\]

	\noindent in a Tarski model would contradict axiom $c_m \neq c_\omega$ from the first family of sentences. Nonetheless, if we consider any finite subset $t \subset T$, then (by interpreting in a Tarski structure with infinite domain, $c_\omega$ the same way as $c_n$ for $n$ bigger than the highest index appearing inside $t$) we can obtain a Tarski model of $t$. Hence, we have an inconsistent $\mathcal{L}_{\omega_1\omega}$-theory for Tarski semantics all whose finite subsets are Tarski consistent. That is, the simplistic notion of compactness fails for the logic $\mathcal{L}_{\omega_1\omega}$ with respect to Tarski semantics (and more generally for any other $\mathcal{L}_{\kappa \lambda}$).\medskip
	
	\ding{95} \textbf{Failure of compactness for Boolean valued semantics.} Let us argue that the same theory shows the failure of compactness for the semantics given by Boolean valued models. Assume $\mathcal{M}$ is a $\bool{B}$-valued model of $T$. Then 
	
	\[
		\Qp{c_n \neq c_\omega} = 1_\bool{B}
	\]

	\noindent for every $n \in \omega$ and 
	
	\[
		\Qp{\bigwedge_{n \in \omega} c_n \neq c_\omega}_\bool{B} = \bigwedge_{n \in \omega} \Qp{c_n \neq c_\omega}_\bool{B} = 1_\bool{B}.
	\]

	\noindent But at the same time
	
	\[
		\neg \Qp{\bigwedge_{n \in \omega} c_n \neq c_\omega}_\bool{B} = \Qp{\neg \bigwedge_{n \in \omega} c_n \neq c_\omega}_\bool{B} = \Qp{\bigvee_{n \in \omega} c_n = c_\omega}_\bool{B} = 1_\bool{B},
	\]

	\noindent a contradiction.\medskip

	\ding{95} Note that if we consider the theory $T'$ obtained from $T$ by adding to $T$ all the sentences given by the
conjunction of any of its finite subset, $T$ and $T'$ are logically equivalent and $T'$ is still Boolean inconsistent even if all its finite subsets are Tarski consistent. It is worth taking into account this trivial observation while parsing the definition of finite conservativity to follow. 
\end{example}

\subsection{Conservative strengthening and finite conservativity} \label{subsec:constrfincns}

We introduce here the key revision of the notion of finite consistency needed to generalize the compactness theorem to $\Linff$.

\subsubsection{What are the subsentences of an $\Linff$-sentence?}
\begin{notation}\label{not:subformulaeconprop}
Given an $\mathrm{L}_{\kappa\lambda}$-sentence $\psi$ its family of \emph{subsentences} is obtained by considering a fresh set of constants $\mathcal{C}$ of size $\lambda$ and all the subsentences obtained by a subformula $\phi(\vec{x})$ in displayed string of 
free variables $\vec{x}$ by substituting to $\vec{x}$ a string of constants $\vec{c}\in(\mathcal{D}\cup\mathcal{C})^{\vec{x}}$.
\end{notation}

There is an ambiguity in the choice of (the size of) $\mathcal{C}$, as an $\mathrm{L}_{\kappa\lambda}$-sentence $\psi$ is also an $\mathrm{L}_{\delta\theta}$-sentence for any $\delta\geq\kappa$ and $\theta\geq\lambda$, hence according to the choice of $\mathcal{C}$ the notion of subsentence of $\psi$ (and the size of the family of subsentence of $\psi$) changes. The ambiguity can be removed by fixing $\kappa,\lambda$ large enough according to the context. Typically, we shall consider $\Linff$-theories consisting of sentences $\psi$ which are all in some fixed $\mathrm{L}_{\kappa\lambda}$ and the notion of subsentence of each such $\psi$ will be computed relative to a fixed $\mathcal{C}$ of size $\lambda$. The context in which we use this notation will clarify how to solve the ambiguity.

The above convention on subsentences is somewhat divergent from the standard terminology, but is natural when taking into account that the Boolean consistency problem for $T$ is equivalent to the Boolean consistency problem for the quantifier free theory $T_{\mathcal{C}}$ (as outlined in Section \ref{subsec:qfversionofT}).

\begin{remark}
Note that a subsentence of $\psi$ relative to the fresh set of constant $\mathcal{C}$
according to Notation \ref{not:subformulaeconprop} is naturally in correspondence with a subformula of $\psi_{\mathcal{C}}$ (according to Definition \ref{def:QECaxiom}) and conversely. Hence, determining which finite sets of subsentences of $\psi$ are Boolean consistent with $\psi$ amounts to study which finite sets of subformulae of  $\psi_{\mathcal{C}}$ are Boolean consistent with $\psi$.
\end{remark}

The following remark (and its proof) are useful only for those readers interested in getting some sharper bounds on the logical/set theoretic complexity of the theories we shall consider in the sequel.
\begin{remark}\label{rem:boundonsubsent}
Let $\kappa$ be a regular uncountable cardinal and $\psi$ be an $\mathrm{L}_{\kappa\lambda}$-sentence for a language $\mathrm{L}$ which has a set of constants $\mathcal{D}$ of size $\delta$. Then the set of subsentences of $\psi$ has size stricly smaller than $\max\bp{\kappa,((\delta+\lambda)^{<\lambda})^+}$.

\end{remark}

\begin{proof}
Note that an $\mathrm{L}_{\kappa\lambda}$-formula on a language of size $\delta$ can be seen as a labelled well-founded tree with branching of size strictly less than $\kappa$ and where the leaves (i.e. its atomic subformulae) can be chosen on the set of atomic formulae of $\mathrm{L}$. Such a well founded tree is an object of size and rank strictly less than $\kappa$, since $\kappa$ is regular and uncountable. On the other hand the label of a node are uniquely determined when the node is associated to an infinitary boolean connective ($\bigwedge, \bigvee,\neg$) and the branching of this node is bounded below $\kappa$, while the label specifies the string of variables which are quantified when the node is associated to a infinitary quantifier ($\exists \vec{x},\forall \vec{x}$) and the node is not splitting in this case (also in case of $\neg$ the node is not splitting). 


Hence, for a given $\mathrm{L}_{\kappa\lambda}$-sentence $\psi$, its set of subformulae has size equal to that of the tree $T$ coding the sentence, with its associated labeling, i.e. some function with domain $T$ and range in  $\lambda^{<\lambda}$.
Note that $T$ has size less than $\kappa$ since it is well-founded and its branching nodes have always less than $\kappa$-many successors.

Each subformula $\phi(\vec{x})$ in displayed free variables of $\psi$ is associated to a unique node and gives the possible subsentences of the form 
$\phi(\vec{c})$ which can be obtained from it by substituting $\vec{x}$ by a string of constants $\vec{c}\in(\mathcal{D}\cup\mathcal{C})^{\vec{x}}$.

The subsentences attached to $\phi(\vec{x})$ form a set of size  bounded by
$(\delta+\lambda)^{\lambda}$. Putting all together, the possible subsentences of $\psi$ form a set whose size is bounded by $\gamma\cdot (\delta+\lambda)^{<\lambda}$, where $\gamma$ is the size of $\psi$. The conclusion follows, since $\gamma<\kappa$.
\end{proof}

\subsubsection{Conservative strenghtening}

\begin{definition}
\label{constre}
	Let $\psi_0$ and $\psi_1$ be $\Linff$-sentences. We say that $\psi_1$ is a \emph{conservative strengthening} of $\psi_0$ if:

	\begin{enumerate}
		\item $\psi_1 \vdash \psi_0$ and
		\item for any finite set $s$ of subsentences of $\psi_0$ (according to Notation \ref{not:subformulaeconprop}), $s \cup \{\psi_0\}$ is Boolean consistent if and only if $s \cup \{\psi_1\}$ is Boolean consistent.
	\end{enumerate}
\end{definition}

To argue for the naturalness and non-ambiguity of the notion of \emph{conservative strengthening}, 
we need to show that if $\psi_0$ is an $\mathrm{L}_{\kappa\lambda}$-sentence, and $\theta\geq\lambda$, the notion of conservative strengthening is independent of the decision whether to compute the notion of subsentence (according to Notation \ref{not:subformulaeconprop}) of $\psi_0$ using the fresh set of constants $\mathcal{C}$ of size $\lambda$ or instead choosing to do it with a larger set of fresh constants of size $\theta$. This gives that the above definition is well posed, even if ambiguity has been left on the size of the fresh set of constants $\mathcal{C}$ used to compute the subsentences of $\psi_0$.
It occurs due to the following:
\begin{fact}
Let $\psi_1\vdash\psi_0$ with $\psi_1$ Boolean consistent and $\psi_0$ an $\mathrm{L}_{\kappa\lambda}$-sentence.
Let also $\mathcal{C}_0$ be a fresh set of constants for $\mathrm{L}$ of size $\lambda$ and 
$\mathcal{C}_1$ be a fresh set of constants for $\mathrm{L}$ of size $\theta\geq\lambda$.
TFAE:
\begin{itemize}
\item
$\psi_1$ is a conservative strenghthening of $\psi_0$ relative to the subsentences of $\psi_0$ computed using $\mathcal{C}_0$ according to Notation \ref{not:subformulaeconprop};
\item
$\psi_1$ is a conservative strenghthening of $\psi_0$ relative to the subsentences of $\psi_0$ computed using $\mathcal{C}_1$ according to Notation \ref{not:subformulaeconprop}.
\end{itemize} 
\end{fact}

In particular  (sticking to the notation of the Fact) a set of fresh constants of size large enough to accomodate the substitutions of all the possible strings of free variables occurring in some finite set $s$ of subformulae of $\psi_0$ is enough to test whether $\psi_1$ is a conservative strengthening of $\psi_0$ according to any possible choice of sets of fresh constants of that or larger size. The proof does not have any special feature, but is long, mainly for notational reasons.

\begin{proof}
	
	\item[\ding{95}]We give in details the proof of the implication from top to bottom. 
	
	\begin{itemize}
		\item Assume $\psi_1$ is a conservative strenghtening of $\psi_0$ relative to the subsentences of $\psi_0$ computed using constants from $\mathcal{C}_0$. We have to prove that $\psi_1$ is a conservative strengthening of $\psi_0$ relative to the subsentences of $\psi_0$ computed using constants from $\mathcal{C}_1$. 
	
		\item Assume $s$ is a finite set of subsentences of $\psi_0$ relative to $\mathcal{C}_1$ such that 
	
	\[
		\psi_0 \wedge \bigwedge s
	\]

	\noindent is Boolean consistent. We need to check that
	
	\[
		\psi_1 \wedge \bigwedge s
	\]

	\noindent is also Boolean consistent. 
	
		\item Let $(c_{j1} : j \in J) \in \mathcal{C}_1^J$ denote the constants from $\mathcal{C}_1$ occuring in $\bigwedge s$. Since $\psi_0$ is an $\mathrm{L}_{\kappa\lambda}$-sentence, $J$ has size strictly less than $\lambda$. Denote $\theta(\vec{v}) = \bigwedge s (\vec{v})$ the $\mathrm{L}_{\kappa\lambda}$-formula such that 
	
	\[
		\bigwedge s = \theta \big( \vec{v} \mapsto (c_{j1} : j \in J) \big).
	\]

	\item By hypothesis there exists a $\bool{B}$-model  for $\mathrm{L}(\mathcal{C}_1)$ of $\psi_0 \wedge \bigwedge s$. Denote it as $\mathcal{M}^{\mathcal{C}_1}$. Let $\mathcal{M}$ be the restriction of $\mathcal{M}^{\mathcal{C}_1}$ to the language $\mathrm{L}$. Fix also $(c_{j0} : j \in J)$ an injective sequence in $\mathcal{C}_0$. We will expand $\mathcal{M}$ to an $\mathrm{L}(\mathcal{C}_0)$-structure $\mathcal{M}^{\mathcal{C}_0}$ such that 
	
	\[
		\Qp{\psi_0 \wedge \theta \big( \vec{v} \mapsto (c_{j0} : j \in J) \big)}^{\mathcal{M}^{\mathcal{C}_0}} = 1_{\bool{B}}.
	\]

	\item Fix an injection $f:\mathcal{C}_0 \to \mathcal{C}_1$ mapping $c_{j0}$ to $c_{j1}$ for all $j\in J$ (it does not matter at all what $f$ does outside this sequence, since no other constants from $\mathcal{C}_1$ appear in $\psi_0 \wedge \theta(c_{j1} : j \in J)$). This is possible since $J$ has size strictly smaller than $\theta$, $\lambda$.

	Define $\mathcal{M}^{\mathcal{C}_0}$ as the expansion of $\mathcal{M}$ to the language $\mathrm{L}(\mathcal{C}_0)$ where the interpretation of every $c \in \mathcal{C}_0$ is given by the interpretation of $f(c)$ in $\mathcal{M}^{\mathcal{C}_1}$. The key point is that the constant $c_{j0}$ is interpreted by the same element in $\mathcal{M}^{\mathcal{C}_0}$ as $c_{j1}$ is in $\mathcal{M}^{\mathcal{C}_1}$. Hence,
	
	\[
		\Qp{\theta(c_{j0} : j \in J)}^{\mathcal{M}^{\mathcal{C}_0}} = \Qp{\theta(c_{j1} : j \in J)}^{\mathcal{M}^{\mathcal{C}_1}}.
	\]

	\noindent Therefore, $\mathcal{M}^{\mathcal{C}_0}$ is a model of 
	
	\[
		\psi_0 \wedge \bigwedge s (c_{j0} : j \in J).
	\]

	\item By hypothesis $\psi_1$ is a conservative strengthening of $\psi_0$ relative to $\mathcal{C}_0$. Hence, 
	
	\[
		\psi_1 \wedge \bigwedge s (c_{j0} : j \in J)
	\]

	\noindent has a $\bool{C}$-model $\mathcal{N}^{\mathcal{C}_0}$. 
	
	\item Now we extend the partial function $f^{-1}$ to a surjection 
$g:\mathcal{C}_1 \to \mathcal{C}_0$.
The very same argument for replacing the constants $(c_{j0} : j \in J)$ with $(c_{j1} : j \in J)$ ensures that we can use $g$ to transform $\mathcal{N}^{\mathcal{C}_0}$ into an $\mathrm{L}(\mathcal{C}_1)$-model $\mathcal{N}^{\mathcal{C}_1}$ of 
	
	\[
		\psi_1 \wedge \bigwedge s
	\]

	\noindent (by setting $c^{\mathcal{N}^{\mathcal{C}_1}}=g(c)^{\mathcal{N}^{\mathcal{C}_0}}$ for all $c\in\mathcal{C}_1$), as was to be shown.
	
	\end{itemize}
	
	\item[\ding{95}] The implication from bottom to top follows the same pattern, by reversing the steps in which we use the functions $f,g$ (i.e. using $g$ to move from $\mathcal{M}^{\mathcal{C}_0}$ to $\mathcal{M}^{\mathcal{C}_1}$ and $f$ to move from $\mathcal{N}^{\mathcal{C}_1}$ to $\mathcal{N}^{\mathcal{C}_0}$).
\end{proof}

	Note that what is crucially used in the above proof is that the free variables occurring in a subformula of $\psi_0$ have size bounded below the minimum of the size of $\mathcal{C}_0,\mathcal{C}_1$; this grants the existence of a bijection between the fresh constants $\bp{c_{jl}:j\in J}$ for $l=0,1$ which substitute free variables of the subformula. This is essential for the above proof to work.

In general if $\psi_1$ is a conservative strengthening of $\psi_0$ and $s$ is a finite set of subsentences of $\psi_0$ consistent with it, the models of $\psi_1\wedge\bigwedge s$ may not have much to share with those of $\psi_0\wedge\bigwedge s$.
Below we outline a simpler case in which $\psi_1$ is a conservative strengthening of $\psi_0$:
\begin{notation}
	Let $\psi_0,\,\psi_1$ be $\Linff$-formulae with $\psi_1$ possibly in a richer signature than $\psi_0$. $\psi_1$ is \emph{strongly conservative} over $\psi_0$ if any model of $\psi_0$ can be expanded to a model of $\psi_1$.
\end{notation}

The above is a natural strengthening of the notion of conservative strengthening (the proof is left to the reader):

\begin{fact}
If $\psi$ is strongly conservative over $\phi$ and $\psi \vdash \phi$, then $\psi$ is a conservative strengthening of $\phi$.
\end{fact}

\subsubsection{Finite conservativity}

\begin{definition}
\label{def:finconservativity}
A theory $\{\psi_i : i \in I\}$ is \textbf{finitely conservative} if:
	\begin{itemize}
		\item at least one $\psi_i$ for $i\in I$ is \emph{Boolean consistent}, and
		\item for every $J_0\subseteq J_1$
		finite and non-empty subset of $I$, $\bigwedge_{i\in J_1}\phi_i$ is a \emph{conservative strengthening} of $\bigwedge_{i\in J_0}\phi_i$.
	\end{itemize}	
\end{definition}
Finite conservativity is a natural strengthening of finite (Boolean) consistency:

\begin{fact}
	Any finitely conservative theory $\{\psi_i : i \in I\}$ is finitely Boolean consistent.
\end{fact}

\begin{proof}
	Let $J \subset I$ be finite. We need to argue $\{\psi_i : i \in J\}$ is Boolean consistent. By hypothesis we can fix $i_0 \in I$ such that $\psi_{i_0}$ is Boolean consistent. By assumption $\bigwedge(\{\psi_i : i \in J\} \cup \{\psi_{i_0}\})$ is a conservative strengthening of $\psi_{i_0}$, which is Boolean consistent. Therefore, so is $\bigwedge(\{\psi_i : i \in J\} \cup \{\psi_{i_0}\})$. In particular $\{\psi_i : i \in J\}$ is Boolean consistent. 
\end{proof}

It is often convenient to reformulate finite conservativity limiting our attention to sets of sentences which are closed under finite conjunctions.
The following Fact is almost self-evident and is left to the reader:

\begin{fact} \label{fac:equivformfincons}
A theory $\{\psi_i : i \in I\}$ is finitely conservative if and only if letting $T$ be the family
\[
\bp{\bigwedge_{i\in u}\psi_i: u \text{ is a finite subset of }I },
\] we have that
\begin{itemize}
		\item at least one $\phi\in T$ is \emph{Boolean consistent},
		\item for every $s$
		finite and non-empty subset of $T$, $\bigwedge s$ is a \emph{conservative strengthening} of $\phi$ for all $\phi\in s$.
	\end{itemize}	
\end{fact}

The following two remarks are worth of attention as they point out the subtle differences between finite conservativity and finite consistency.

\begin{remark}
	Consider the setting from Example \ref{faicom}. Let us argue that $T$ is not finitely conservative. Let $t \subseteq T$ be the finite subset given by the two sentences
	
	\[
		\bigvee_{n \in \omega} c_\omega = c_n \ \ \text{and} \ \ c_\omega \neq c_0.
	\]

	\noindent Let us argue that $\bigwedge t$ is not a conservative strengthening of $\bigvee_{n \in \omega} c_\omega = c_n$. In order to do so, we need to find a subsentence of $\bigvee_{n \in \omega} c_\omega = c_n$ that is consistent with $\bigvee_{n \in \omega} c_\omega = c_n$, but is not consistent with $\bigwedge t$.
	
	We have that $c_\omega = c_0$ is a consistent subsentence of $\bigvee_{n \in \omega} c_\omega = c_n$. Nonetheless, $c_\omega = c_0$ is not consistent with $\bigwedge t$, since $c_\omega \neq c_0 \in t$. Thus, $\bigwedge t$ is not a conservative strengthening of $\bigvee_{n \in \omega} c_\omega = c_n$, and $T$ is not finitely conservative.  
\end{remark}

	Also: 
\begin{remark}
	There is no reason a priori to expect the following to be true: 
	\begin{quote}
		\emph{
			Assume that a theory $T=\bp{\phi_i:i\in I}$ which is not closed under finite conjunctions is such that for all finite subsets $J$ of $I$, $\bigwedge_{i\in J} \psi_i$ is a conservative strengthening of each $\psi_i$ for $i \in J$. 
		Then $T$ is finitely conservative.
		}
	\end{quote}
	The problem is that there could be $J_0\subset J_1$ finite subsets of $I$ such that $\bigwedge_{i\in J_1}\psi_i$ is a conservative strengthening of each $\psi_i$ for $i\in J_1$, but not of $\bigwedge_{i\in J_0}\psi_i$. 
\end{remark}

Note in contrast that a theory is finitely consistent if and only if its closure under logical consequence is.

\subsection{The Compactness Theorem for $\Linff$}

\begin{theorem}[Boolean Compactness for $\Linff$] \label{conservative}
	Assume $T$ is an $\Linff$-theory such that at least one of its sentences is Boolean consistent. Then $T$ is finitely conservative if and only if $\bigwedge T$ is a conservative strenghtening of $\bigwedge s$ for any $s$ non-empty finite subset of $T$.
\end{theorem}

\begin{proof}
	Only the direction 

	\begin{quote}
		\emph{ 
			$T$ is finitely conservative $\Rao$ $\bigwedge T$ is a conservative strengthening of $\bigwedge s$ for any $s$ finite non-empty subset of $T$
		} 
	\end{quote}
	requires a detailed argument.\medskip

	Assume $T = \{\psi_i : i \in I\}$ is finitely conservative (and thus also finitely consistent). W.l.o.g we may assume that $T$ contains as elements the conjunctions of its finite subsets (by Fact \ref{fac:equivformfincons}). 
Hence under this assumption on $T$ we need to show that 
\begin{quote}
		\emph{ 
			$T$ is finitely conservative $\Rao$ $\bigwedge T$ is a conservative strengthening of $\psi_i$ for any $i\in I$.
		} 
	\end{quote}

Let
		
	\[
		\Psi = \bigwedge_{i \in I} \psi_i.
	\]
	Let $\kappa$ be a regular cardinal such that $\Psi$ is an $\mathrm{L}_{\kappa\kappa}$-sentence and $\mathcal{D}$ be the set of constants in $\mathrm{L}$.
	\begin{enumerate} 
	\setlength\itemsep{1em}
		\item Consider the language $\mathrm{L}(\mathcal{C})$ with $\mathcal{C}$ a set of constants of size $\kappa$ which is fresh for the language $\mathrm{L}$.
The family $S$ is given by the sets $\{\Psi\} \cup t$ such that:
	
			\begin{itemize}
				\item $t$ is a finite set of sentences from $\mathrm{L}(\mathcal{C})$,
				\item $t$ is Boolean consistent,
				\item there exists $i_t \in I$ such that:
					\begin{itemize}
						\item $\psi_{i_t} \in t$,
						\item $\theta$ is a subsentence of $\psi_{i_t}$ for the language $\mathrm{L}$ for each $\theta \in t$ (according to Notation \ref{not:subformulaeconprop} and using -as in there- the constants in $\mathcal{C}\cup\mathcal{D}$ to determine the possible subsentences of $\psi_{i_t}$).
					\end{itemize}
			\end{itemize}

		\item Let us first show that if $S$ is a consistency property as in Def. \ref{def:ConProInf}
(with $\mathcal{C}$ the set of constants to which Clauses \ref{def:conspropInd5} (Ind.5) and \ref{def:conspropStr3} (Str.3) of that Definition apply), then $\Psi$ is a conservative strengthening of $\psi_i$ for any $i\in I$.
			
			\medskip
			\begin{itemize}
				\item Let $s$ be a finite set of subsentences of $\psi_i$. We need to show that if $s \cup \{\psi_i\}$ is Boolean consistent, then so is $s \cup \{\Psi\}$. Toward this aim, note that if $s \cup \{\psi_i\}$ is Boolean consistent, then $\{\Psi\} \cup s \cup \{\psi_i\} \in S$, as $i$ is exactly the $i_t$ for $t = s\cup\bp{\psi_i}$ witnessing that $\{\Psi\} \cup t \in S$. By the Boolean Model Existence Theorem \ref{ModExiThe}, this implies the existence of a Boolean valued model of $\{\Psi\} \cup t$, which in particular is a model of $\{\Psi\} \cup s$.
			\end{itemize}

		\item Now we show that $S$ is a consistency property. 
			
			\medskip
			\begin{itemize}
				\item Suppose $\{\Psi\} \cup t \in S$. We need to argue that for each clause in the definition of consistency property \ref{def:ConProInf}, the relevant subsentence under consideration in the clause belongs to an extension of $\{\Psi\} \cup t$ in $S$.
				\item If we deal with sentences in $t$, then $t$ being Boolean consistent allows to deal with all clauses in the definition of a consistency property (here we take advantage of our choice of $\mathcal{C}$ at the same time as the fresh set of constants used to compute the subsentences of some $\psi_i$ as in Notation \ref{not:subformulaeconprop} and as the fresh set of constants used to verify Clauses \ref{def:conspropInd5} (Ind.5) and \ref{def:conspropStr3} (Str.3) of Def. \ref{def:ConProInf} for $S$). 
				\item Therefore, we only need to deal with the case of $\Psi$, i.e. the only sentence that is not inside of $t$ (and which is an infinitary conjunction).
				\item Fix $j \in I$. We need to argue that there is $r \in S$ such that $\{\Psi\} \cup t \cup \{\psi_j\} \subset r$. 
				\item By definition of $S$, there is $p=i_t \in I$ such that $\psi_p \in t$ and all $\theta$ in $t$ are subsentences of $\psi_p$.
				\item Since the family $\{\psi_i : i \in I\}$ is finitely conservative, we have that $\psi_p\wedge\psi_j$ is a conservative strengthening of $\psi_p$.
				\item Since the family $\{\psi_i : i \in I\}$ is closed under finite conjunctions, we have that $\psi_p \wedge \psi_j = \psi_k$ for some $k\in I$. 
				\item Since $\psi_k$ is a conservative strengthening of $\psi_p$, for any $s$ finite set of $\psi_p$-subsentences, $\psi_k \wedge \bigwedge s = \psi_p\wedge \psi_j \wedge \bigwedge s$ is Boolean consistent if and only if so is $\psi_p \wedge \bigwedge s$. 
				\item Since $t$ is Boolean consistent, $\psi_p \in t$, and all sentences in $t$ are subsentences of $\psi_p$, we have that 
				
					\[
						\bigwedge t=\bigwedge(\bp{\psi_p}\cup t)=\psi_p\wedge\bigwedge t
					\]

					\noindent is Boolean consistent. Therefore, $s = \{\psi_j,\psi_k\} \cup t$ is Boolean consistent. 
				\item Finally, as all sentences in $s$ are subsentences of $\psi_k$, $\{\Psi\} \cup t \cup \bp{\psi_j,\psi_k}$ belongs to $S$.
			\end{itemize}
	\end{enumerate}
\end{proof}

Since the models provided by the Boolean Compactness Theorem are obtained as applications of the Model Existence Theorem, they still retain that every element in the model equals the interpretation of some constant from the language.

\subsection{A non-trivial application of the Boolean compactness theorem}\label{subsec:appcomp}

We show below that if an $\Linf$-sentence $\phi$ is Boolean consistent, then we can produce models of it which realize plenty of other non-trivial constraints. Actually we expect that the Boolean compactness theorem will turn out to be a very powerful tool to produce models which realize certain types and omit others. A much deeper use of the Boolean compactness Theorem is given in \cite[Ch. 4-5]{SantiagoPhd}, where it appears as an essential ingredient of a new proof of Asperò and Schindler's result that Woodin's axiom $(*)$ follows from $\mathsf{MM}^{++}$ \cite{ASPSCH(*)}.

The motivation for the example presented in this section comes from the set theoretic method of forcing, however the arguments do not involve any familiarity with forcing, just with the Boolean Compactness Theorem. The reader wishing to familiarize more with this example is invited to parse through \cite[Ch. 5]{SantiagoPhd}.

\begin{definition} \label{S_psi}
	Let $\phi$ be a Boolean consistent $\Linf$-sentence. Fix an infinite and countable set $\mathcal{C}$ of fresh  constants for $\phi$. The partial order (forcing notion)  $S_\phi$ has as domain the set of $s$ such that
	
		\begin{enumerate}					
			
			\setlength\itemsep{0.5em}

			\item \label{1} $|s| < \omega$,
			
			\item \label{2} all elements in $s$ are proper subsentences of $\phi$ or (negated) atomic $\mathrm{L}(\mathcal{C})$-sentences in which only finitely many constants from $\mathcal{C}$ appear, and
			
			\item \label{3} $\bigwedge (s\cup\bp{\phi})$ is Boolean consistent, 			
		\end{enumerate}		
		
	\noindent ordered by reverse inclusion, i.e.  $s \leq t \Lrao t \subseteq s$.
\end{definition}

Let $\phi$ be a Boolean consistent $\Linf$-sentence. Recall that a set $D \subseteq S_\phi$ is dense if for all $s \in S_\phi$ there exists $t \in D$ such that $s \subseteq t$. 
Since all elements in $S_\phi$ are made up of $\Linf$-sentences, for every subset $D \subset S_\phi$ the following is an $\Linf$-sentence:

	\[
		\bigvee_{s \in D} \bigwedge s.
	\]
	
Now note the following:

\begin{itemize} 
	\item A filter $G$ on $S_\psi$ is such that $\Sigma_G=\bigcup G$ is a finitely consistent theory.
	\item Given a \emph{maximal} filter $G$ on $S_\psi$, one can define a Tarski $\mathrm{L}(\mathcal{C})$-structure $\mathcal{M}_G$ whose elements are the classes $[c]_G=\bp{d: d=c\in \Sigma_G}$ for $c$ a constant symbol of $\mathrm{L}$, and where $R^{\mathcal{M}_G}([c_1]_G,\dots,[c_n]_G)$ holds if and only if $R(c_1,\dots,c_n)\in \Sigma_G$ for $R$ an $n$-ary relation symbol of $\mathrm{L}$.
	\item It is not hard to check that for a maximal filter $G$ on $S_\psi$,
$G\cap D$ is non-empty if and only if $\mathcal{M}_G\models\bigvee_{s \in D} \bigwedge s$.
\end{itemize}
 
We want to show that
	\begin{equation}\label{eqn:vgenfilt}
		\phi\wedge\bigwedge_{D \subseteq S_\phi \text{ dense}} \bigg( \bigvee_{s \in D}\bigwedge s \bigg)
	\end{equation}
is a Boolean consistent sentence whenever $\phi$ is.\medskip	
	
Note that a Tarski model $\mathcal{M}$ of (\ref{eqn:vgenfilt}) induces what in the set-theoretic terminology is called a $V$-generic filter for $S_\phi$ by looking at the $s\in S_\phi$ which hold in $\mathcal{M}$. Such Tarski models cannot exist in the ground universe $V$ whenever $S_\phi$ does not have minimal elements, but they can exist in an extension of the universe obtained by forcing with $S_\phi$. Nonetheless, we now argue that the existence of a (mixing) Boolean valued model of (\ref{eqn:vgenfilt}) can be obtained by an application of the Boolean compactness theorem.

\begin{lemma}
	Let $\phi$ be a Boolean consistent $\Linf$-sentence. Then 

	\[
		\{\phi\} \cup \{ \bigvee_{s \in D}\bigwedge s : D \subset S_\phi \text{ dense}\}
	\] 
	
	\noindent is Boolean consistent. Furthermore, its conjunction is a conservative strengthening of $\phi$.
\end{lemma}

\begin{proof}
	We show that $T$, the closure under finite conjunctions of 
	
	\[
		\{\phi\} \cup \{ \bigvee_{s \in D}\bigwedge s : D \subset S_\phi \text{ dense}\},
	\] 
		
	\noindent is finitely conservative. Since $\phi$ is Boolean consistent and belongs to $T$, we only need to deal with the third condition in the definition.\medskip
	
	Let $D_1,\ldots,D_k,\ldots,D_m$ be dense sets. Consider the sentence 
	
	\[
		\phi(D_1,\ldots,D_m):= \phi \wedge \bigvee_{s \in D_1} \bigwedge s \wedge \ldots \wedge \bigvee_{s \in D_k} \bigwedge s \ldots \wedge \bigvee_{s \in D_m} \bigwedge s
	\]

	\noindent and $t$ a finite set of subsentences from
	
	\[
		\phi(D_1,\ldots,D_k):= \phi \wedge \bigvee_{s \in D_1} \bigwedge s \wedge \ldots \wedge \bigvee_{s \in D_k} \bigwedge s
	\]

	\noindent consistent with $\phi(D_1,\ldots,D_k)$. We need to argue that $t$ is also Boolean consistent with $\phi(D_1,\ldots,D_m)$. Each element in $t$ is either a subsentence of $\phi$, or 
$\bigvee_{s \in D_i} \bigwedge s$, or $\bigwedge s$ for some $s \in D_i$ for some $i \leq k$.

 	Hence, $t$ can be decomposed as 

	\[
		t' \cup \{\bigwedge s_i:i \in I\}\cup\bp{ \bigvee_{s \in D_j} \bigwedge s: j \in J}
	\] 
	
	\noindent for $I,J\subseteq \bp{1,\dots,k}$ with $t' \in S_\phi$ and $s_i \in D_i$ for $i \in I$. This gives that $\phi\wedge t'\wedge\bigwedge\{\bigwedge s_i:i \in I\}$ is Boolean consistent, hence

	\[
		t'\cup \bigcup\bp{s_i:i\in I} =t^*
	\]
	 
	\noindent is in $S_\phi$. 

Now, using the density of each $D_{k+1},\ldots,D_m$ and of the elements of $\bp{D_j:j\in J}$ in $S_\phi$, letting
$J\cup\bp{k+1,\dots, m}=\bp{i_0,\dots,i_l}$, and  $r_0=t^*$, we can build a decreasing chain $r_h$ for every $h=0,\dots,l$  with $r_h \in D_{i_h}$. Then $r_l \leq t^*$ is in $S_\phi$ and witnesses that $t$ is Boolean consistent with $\phi(D_1,\dots,D_m)$.

	Since $T$ is finitely conservative, the Boolean Compactness Theorem \ref{conservative} ensures that $\bigwedge T$ is Boolean consistent and is a conservative strengthening of $\phi$.	
\end{proof}

\section{Finite conservativity versus Boolean consistency and Tarski consistency}\label{FinConBooConTarCon}
 
This section compares the notion of finite conservativity with those of Boolean and Tarski consistency. For countable $\mathrm{L}_{\omega_1\omega}$-theories the three notions are essentially equivalent, while for uncountable $\mathrm{L}_{\infty\omega}$-theories finite conservativity is equivalent to Boolean consistency, and, hence, weaker than Tarski consistency. 
Note that there is no sensible notion of completeness for Boolean consistent theories which are not Tarski consistent: it is the essential feature of a Boolean valued model which is not a Tarski model that of having some $\phi$ whose Boolean truth value in the model is neither true nor false, note however that $\phi\vee\neg\phi$ has maximum truth value in the model. In such a set-up, the Boolean theory  validated by the model (meant as the family of sentences evaluated as the maximum of the Boolean algebra) cannot be complete.

We close the paper listing some natural questions on the ways to establish the actual equivalence of Boolean consistency with finite conservativity which we are not able to answer.

\subsection{Syntactically complete theories vs Boolean consistent theories}

The comparison between Boolean consistency and finite conservativity is more delicate than it might seem at first sight. Logically speaking they are equivalent, as the Boolean Compactness Theorem \ref{conservative} shows that finite conservativity implies Boolean consistency, while if $T$ is a Boolean consistent theory, then $\bp{\bigwedge T}$ is finitely conservative and logically equivalent to $T$. 

Nonetheless, if $T$ is an $\mathrm{L}_{\kappa\lambda}$-theory of size at least $\kappa$,
$\bigwedge T$ is not an $\mathrm{L}_{\kappa\lambda}$-sentence. Hence, the theory $\bp{\bigwedge T}$ can be seen to be more complex than $T$, raising the question: if $T$ is a Boolean consistent $\mathrm{L}_{\kappa\lambda}$-theory, is there an $\mathrm{L}_{\kappa\lambda}$-theory $T^*$ logically equivalent to $T$ and finitely conservative?

We provide a positive answer to this question for syntactically complete $\mathrm{L}_{\kappa\lambda}$-theories. Note that this does not solve the issue for Boolean consistent theories, as it might not be at all possible to extend a Boolean consistent theory to a syntactically complete one, contrary to the case of Tarski consistency.

\begin{definition}
	Let $\kappa$ be a regular infinite cardinal and $T$ be an $\mathrm{L}_{\kappa\lambda}$-theory. We say that $T$ is coherent (with respect to a proof system for $\mathrm{L}_{\infty\infty}$) if $T$ does not derive a contradiction. We say that $T$ is $(\kappa,\lambda)$-syntactically complete if for every $\mathrm{L}_{\kappa\lambda}$-sentence $\phi$, either $T \vdash \phi$ or $T \vdash \neg \phi$.
\end{definition}

When dealing with $(\kappa,\lambda)$-syntactically complete theories, we usually assume they are closed under logical deduction. Hence, for every $\mathrm{L}_{\kappa\lambda}$-sentence $\phi$, either $\phi \in T$ or $\neg \phi \in T$. Note that for $\kappa=\lambda=\omega$, the above notion is (an apparent strengthening of\footnote{It is the usual notion of syntactically complete and coherent first order theory, if one can show that provable sequents (according the proof system given in Section \ref{subsec:proofsystem}) consisting of first order formulae admit a proof where all formulae appearing in some of its sequents are first order. This is the case once one proves the cut-elimination theorem also for this proof system (a result we have ignored and we do not use in the present paper).}) the usual notion of coherent and syntactically complete first order theory.

\begin{theorem} \label{thm:comthe}
	Let $\kappa$ be a regular infinite cardinal and $T$ be a coherent and $(\kappa,\gamma)$-syntactically complete $\mathrm{L}_{\kappa\gamma}$-theory of size $\lambda$ for a language $\mathrm{L}$ whose set of constants has size $\delta$. Then there exists $T^*$ an $\mathrm{L}_{\theta\gamma}$-theory for 
$\theta=\max\bp{((\delta+\gamma)^{<\gamma})^+,\kappa}$ of size $\lambda$ such that 
	
	\begin{itemize}
		\item $T^*$ is logically equivalent to $T$, and
		\item $T^*$ is finitely conservative.
	\end{itemize}

Furthermore, if $\kappa>(\delta+\gamma)^{<\gamma}$ is regular and uncountable, then $T^*$ is an $\mathrm{L}_{\kappa\gamma}$-theory  which is a subset of $T$.
\end{theorem}

The reader not interested in the cardinal bounds on the language can ignore the parts of the proof related to these computations.
\begin{proof}
	\emph{}
	\begin{enumerate}
		\item For each $\phi \in T$, let $\{\theta_i:i\in I_\phi\}$ be the family of subsentences of $\phi$ and consider $\phi^*$ to be
	
		\begin{gather*}
			\phi \ \wedge \ \bigwedge \{ \neg \theta_i: i \in I_\phi
			\text{ and } \neg \theta_i\in T\} \ \wedge \ \bigwedge 
			\{\theta_i: i \in I_\phi
			\text{ and } \theta_i\in T\}. 
		\end{gather*}

		\noindent Then $\phi^*$ is in $\mathrm{L}_{\theta\gamma}$ as there are less than $\theta$-many subsentences of $\phi$ (by Remark \ref{rem:boundonsubsent}, as $((\delta+\gamma)^{<\gamma})^++\kappa\leq\theta$). If $\theta=\kappa$, $\phi^*$ is still an $\mathrm{L}_{\kappa\gamma}$-sentence.
		
		Let us argue that $T^*$, the closure under finite conjunctions of 
		\(
			\{ \phi^* : \phi \in T \},
		\)
		\noindent is the theory we are searching for. 
		\item First, we check the equivalence between $T$ and $T^*$.
		
		\begin{itemize}
			\item For every $\phi \in T$, we have that $T$ proves all the conjuncts of $\phi^*$ by its very definition. Hence, $T$ proves $T^*$. Furthermore, if $T^*$ is an $\mathrm{L}_{\kappa\gamma}$-theory, it is a subset of $T$ as $T$ is syntactically complete for $\mathrm{L}_{\kappa\gamma}$.
			\item At the same time, $T^*$ proves every $\phi \in T$, hence they are equivalent. 
		\end{itemize} 

		\item We now prove the first two conditions for being finitely conservative.
		
		\begin{itemize}
			\item First, since $T$ is coherent and $T^*$ is logically equivalent to it, then $T^*$ is coherent and there is at least one sentence from $T^*$ that is consistent.
			\item Also, $T^*$ is closed under finite conjunctions by its very definition.
		\end{itemize}

		\item Now we establish the third crucial property of finite conservativity. That is, if 
		
		\[
			\{ \bigwedge \{\phi_k^* : k \in K_1\}, \ldots,\bigwedge \{\phi_k^* : k \in K_n\} \}
		\]

		\noindent is a finite subset of $T^*$, and $\bigwedge \{\phi_k^* : k \in K_i\}$ is one of its elements, and a finite set $t$ of subsentences of 
		
		\[
			\bigwedge \{\phi_k^* : k \in K_i\}
		\]

		\noindent is consistent with $\bigwedge \{\phi_k^* : k \in K_i\}$, then $t$ is also consistent with 
		
		\[
			\bigwedge \bigg( \bigwedge \{\phi_k^* : k \in K_1\} \wedge \ldots \wedge \bigwedge \{\phi_k^* : k \in K_n\} \bigg).
		\]

		\noindent Consider an enumeration 

		\[
			\bp{\phi^*_1,\dots,\phi^*_m} = \bigcup_{l=1}^n\bp{\phi^*_j: j \in K_l}
		\]

		\noindent such that 
		
		\[
			\bp{\phi^*_1,\dots,\phi^*_k}=\bp{\phi^*_j: j \in K_i}.
		\]

		\item First, notice that if a sentence in $t$ is one from the first three levels (blue) of the decomposition tree of $\phi_1^*\wedge\dots\wedge\phi_k^*$ below, then it is automatically a logical consequence of $\bp{\phi^*_1,\dots,\phi^*_k}$ and of $\bp{\phi^*_1,\dots,\phi^*_m}$, as the only connective labelling a blue node is the conjunction. Thus, we can assume $t$ only contains subsentences of $\phi_1^* \wedge \dots \wedge \phi_k^*$ from the fourth level or below (green). 
			
			
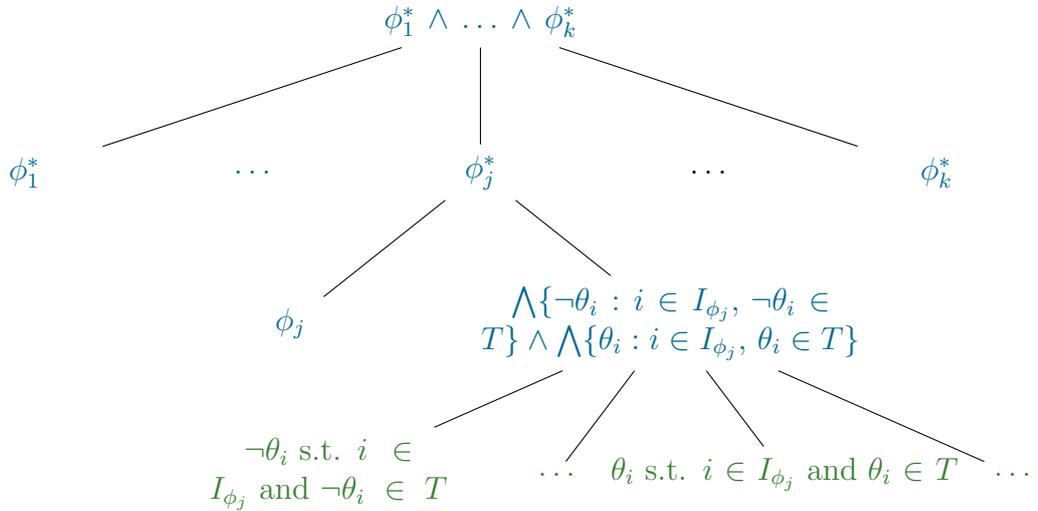
\begin{figure}[H]
				\centering
				\begin{tikzpicture}
					[
						level distance=2cm,
						level 1/.style={sibling distance=3cm},
						level 2/.style={sibling distance=5cm},
						level 3/.style={sibling distance=3cm}, 
						every node/.style={text width=5cm, align=center}
					]
					\node[text=MidnightBlue] {$\phi_1^* \wedge \ldots \wedge \phi_{k}^*$}
						child[text=MidnightBlue] {node {$\phi_1^*$}}
						child {node[text=MidnightBlue] {$\ldots$} edge from parent[draw=none]}
						child {
							node[text=MidnightBlue] {$\phi_j^*$}
							child {node[text=MidnightBlue] {$\phi_j$}}
							child {
								node[text=MidnightBlue] {$\bigwedge \{\neg \theta_i: i\in I_{\phi_j},\,\neg \theta_i  \in T  \}\wedge\bigwedge \{\theta_i: i\in I_{\phi_j},\, \theta_i  \in T  \}$}
								child{ node[text=OliveGreen]{$\neg \theta_i \text{ s.t. } i\in I_{\phi_j} \text{ and } \neg\theta_i \in  T $} }
child{ node[text=OliveGreen]{$\ldots$}}
child{ node[text=OliveGreen]{$\theta_i \text{ s.t. }  i\in I_{\phi_j} \text{ and }  \theta_i \in T $} }
								child{ node[text=OliveGreen]{$\ldots$}}
							}
						}
						child {node {$\ldots$} edge from parent[draw=none]}
						child {node[text=MidnightBlue] {$\phi_{k}^*$}};
				\end{tikzpicture}
				\end{figure} 
				
	\item Now, every subsentence $\psi$ of $\phi_1^* \wedge \dots \wedge \phi^*_k$ labelled by a green node is:
		
		\begin{itemize} 
			\item either a subsentence $\theta_i$ of $\phi_j$ for some $j=1,\dots,k$ and $i\in I_{\phi_j}$, 
			\item or the negation of a subsentence $\theta_i$ of $\phi_j$ 
for some $j=1,\dots,k$ and $i\in I_{\phi_j}$.			
		\end{itemize}
	
	Furthermore, even if $\mathrm{L}$ has size greater or equal to $\kappa$, we have that $\neg \theta_i, \theta_i$ are $\mathrm{L}_{\kappa \gamma}$-sentences, since each $\theta_i$ is a subsentence of an $\mathrm{L}_{\kappa\gamma}$-sentence. Hence, since $T$ is $(\kappa,\gamma)$-syntactically complete, $T$ proves $\theta_i$ or $\neg \theta_i$.
	
	\item We have all the tools to finish the proof. Assume by contradiction that there are $\phi_1,\dots,\phi_m$ in $T$, $0<k<m$ and $s=\bp{\eta_1,\dots,\eta_l}$ a finite set of subsentences of  $\phi_1^*\wedge\dots\wedge\phi_k^*$ (for some $l\in\mathbb{N}$) whose conjunction is consistent with $\phi_1^*\wedge\dots\wedge\phi_k^*$, but not with 
$\phi_1^*\wedge\dots\wedge\phi_m^*$. 
Then 
		
		\[
			T \vdash \bigwedge_{i=1}^m \phi_i^* \vdash \bigvee_{i=1}^l \neg \eta_i.
		\]

		\noindent Since $T$ is complete for $\mathrm{L}_{\kappa \gamma}$, we must have that at least one of the $\neg \eta_i$ is in $T$. But then for some $j$  and $l \in I_{\phi_j}$ we have that $\eta_i=\theta_l$ is a subsentence of $\phi_j$ or $\eta_i=\neg \theta_l$ with $\theta_l$ a subsentence of $\phi_j$. In either cases we have that $\neg \eta_i$ (being in $T$) is (logically equivalent to) one of the conjuncts of $\phi^*_j$ for some $j=1,\dots,k$.

	This contradicts that $\bp{\eta_1,\dots,\eta_l}\cup\bp{\phi_1^*,\dots,\phi^*_k}$ is Boolean consistent.

	We reached a contradiction.
	\end{enumerate}
\end{proof}

\subsection{Tarski consistent theories versus finitely conservative ones}

Tarski consistency implies Boolean consistency, since the set of truth values $\{0,1\}$ can be seen as a Boolean algebra. We show that Tarski consistency is also stronger than finite conservativity on a level by level stratification of $\mathrm{L}_{\infty\lambda}$ as the union of its fragments $\mathrm{L}_{\kappa\lambda}$. 

\begin{proposition}
	Let $T$ be an $\mathrm{L}_{\kappa\lambda}$-theory in a language of size less than $\kappa=\kappa^{<\kappa}\geq\lambda$ which is Tarski consistent as witnessed by a model of size less than $\kappa$. Then there exists a $(\kappa,\lambda)$-syntactically complete $\mathrm{L}_{\kappa\lambda}$-theory $T^*$ stronger than $T$ and finitely conservative.
\end{proposition}

\begin{proof}
	Denote by $\mathcal{M}$ a Tarski model of $T$ with domain $M$ of size less than $\kappa$, and by $T'$ the $(\mathrm{L} \cup M)_{\kappa\lambda}$-theory of $\mathcal{M}$, where $M$ is added as a fresh set of constants to $\mathrm{L}$ and each $m\in M$ is intepreted by itself in the natural expansion of $\mathcal{M}$ to an $\mathrm{L}\cup\bp{M}$-structure. $T'$ is a coherent and $(\kappa,\lambda)$-syntactically complete theory. Hence, Theorem \ref{thm:comthe} (and the cardinal assumptions on $\kappa,\lambda$, and $\mathrm{L}$) ensure the existence of an $\mathrm{L}_{\kappa\lambda}$-theory $T^*$, logically equivalent to $T'$ and finitely conservative. $T^*$ is stronger than $T$ and finitely conservative.	
\end{proof}

\subsection{Boolean Compactness for $\Linff$ generalizes first order compactness}

We can now show that the Boolean Compactness Theorem \ref{conservative} is truly a generalization to $\Linff$ of the compactness theorem for first order logic.

\begin{corollary} \label{thm:concompfirord}
	Let $T$ be a finitely Tarski consistent first order $\mathrm{L}$-theory. Then there is a finitely conservative $\Linf$-theory $T^*$ which is stronger than $T$; hence $T$ is Tarski consistent.
\end{corollary}

\begin{proof}
	Expand $T$ to an $(\omega,\omega)$-syntactically complete and coherent first order theory with a countable set of constants which we denote $T_0$. For this, first add a countable set of constants to the language and note that the theory $T$ is still finitely consistent for the expanded language; then we apply the correctness theorem to get that $T$ is coherent relative to the proof system of Section \ref{subsec:proofsystem} and thus can be extended to a maximal coherent first order (i.e. $\mathrm{L}_{\omega\omega}$) theory $T_0$ for that proof system (by a standard argument rooted in the axiom of choice or its variants). By Theorem \ref{thm:comthe} (applied to the coherent and $(\omega,\omega)$-syntactically complete theory $T_0$) there exist an $\mathrm{L}_{\infty\omega}$-theory $T_1$ (which is not necessarily first order) finitely conservative and logically equivalent to $T_0$. $T_1$, and hence $T_0$, admits a Boolean valued model with the mixing property, by the Boolean Compactness Theorem \ref{conservative}. Since $T_0$ is a first order theory, the equivalence between Tarski consistency and Boolean  consistency for first order logic (see \cite[Fact 9.4]{MatteoJuanBoolean}) produces a Tarski model of $T_0$, which in particular is a model of $T$.
\end{proof}

\subsection{Open questions} \label{sec:opeque}

We have shown that if $T$ is an $\Linff$-theory, then the following sequence of implications holds:

\begin{center}
	$T$ is Tarski consistent \\
	$\Downarrow$ \\
	there exists $T^*$ stronger than $T$ and finitely conservative \\
	$\Updownarrow$ \\
	$T$ is Boolean consistent.
\end{center}

\begin{question}
Assume $\kappa,\lambda \geq \omega$ are regular cardinals, and $T$ is an $\mathrm{L}_{\kappa\lambda}$-theory.
	Can the implication 
	
	\begin{center}
		there exists $T^*$ stronger than $T$ and finitely conservative \\
		$\Uparrow$ \\
		$T$ is Boolean consistent
	\end{center}
	
	\noindent be proved by providing $T^*$ which is also an $\mathrm{L}_{\kappa\lambda}$-theory?

\end{question}
Thm. \ref{thm:comthe} gives a positive answer only when $T$ is $(\kappa,\lambda)$-syntactically complete for some $\kappa=\kappa^{<\kappa}\geq\lambda$ of size greater than $\mathrm{L}$.
Otherwise, we only have the trivial solution given by $T^*=\bp{\bigwedge T}$ which is not an $\mathrm{L}_{\kappa\lambda}$-theory when $T$ has size at least $\kappa$.

	There is also a natural connection between Boolean compactness and Barwise compactness.
Recall that a set is admissible if it is transitive and it is a model of Kripke-Platek set theory\footnote{Kripke Platek set theory is obtained by the usual axioms of $\bool{ZF}$ removing powerset, choice, changing replacement with the Collection Principle for $\Sigma_1$-definable relations, and restricting the separation axiom to hold just for $\Delta_0$-formulae.}.

\begin{definition}
	Let $T$ be a countable $\mathcal{L}_{\omega_1\omega}$-theory. $T$ is of \textbf{Barwise-type} if there exists an admissible set $\mathcal{A}$ such that:
	
	\begin{itemize}
		\item $T$ is a $\Sigma_1$-definable class in $\mathcal{A}$, and
		\item every $t \subseteq T$ such that $t \in \mathcal{A}$ has a Tarski model. 
	\end{itemize}
\end{definition}

\begin{theorem}[Barwise] \label{thm:Barwise}
	Let $T$ be an $\mathcal{L}_{\omega_1\omega}$-theory of Barwise-type. Then $T$ is Tarski consistent.
\end{theorem}

Hence, if $T$ is of Barwise-type, it is Tarski consistent and there exists an $\mathrm{L}_{\omega_1\omega}$-theory $T^*$ stronger than $T$ and finitely conservative.

\begin{question}
	Given a Barwise-type theory $T$, is there a proof of the existence of an $\mathrm{L}_{\omega_1\omega}$-theory $T^*$ which is finitely conservative and stronger than $T$ that does not rely on Barwise's theorem and the observation that Tarski complete $\mathrm{L}_{\omega_1\omega}$-theories are logically equivalent to a finitely conservative $\mathrm{L}_{\omega_1\omega}$-theory?
\end{question}

\bibliographystyle{plain}
	\bibliography{Biblio}

\end{document}